\documentclass[cupthm,crop,info]{CUP-JNL-ETS2}%

\usepackage{graphicx,multicol,multirow,amsmath,amssymb,amsfonts,enumerate,footmisc,mathscinet,mathrsfs, bbm,rotating,appendix,amsthm,lipsum}
\usepackage[utf8]{inputenc}
\usepackage[numbers]{natbib}
\usepackage[pagewise]{lineno}
\theoremstyle{cupplain}
\newtheorem{theorem}{Theorem}[section]
\newtheorem{lemma}[theorem]{Lemma}
\newtheorem{corollary}[theorem]{Corollary}
\newtheorem{question}[theorem]{Question}
\newtheorem{prop}[theorem]{Proposition}
\newtheorem{conjecture}[theorem]{Conjecture}
\theoremstyle{cupdefinition}
\newtheorem{definition}{Definition}[section]
\theoremstyle{cupremark}
\newtheorem{remark}[theorem]{Remark}

\theoremstyle{cupproof}
\newtheorem{proof}{Proof}

\numberwithin{equation}{section}

\newcommand\numberthis{\addtocounter{equation}{1}\tag{\theequation}}
\jyear{2022}

\begin{document}

\begin{Frontmatter}

\title{A generalization of van der Corput's difference theorem with applications to recurrence and multiple ergodic averages}

\author{\gname{Sohail} \sname{Farhangi}}

\address{\orgdiv{Department of mathematics and computer science}, \orgname{University of Adam Mickiewicz}, \orgaddress{\city{Pozna\'n}, \postcode{61-712}, \state{Greater Poland Voivodeship},  \country{Poland}}\\ (\email{sohail.farhangi@gmail.com})}

\maketitle

\authormark{S. Farhangi}
\titlemark{A generalization of van der Corput's difference theorem and applications}

\begin{abstract}
We prove a generalization of van der Corput's difference theorem for sequences of vectors in a Hilbert space. This generalization is obtained by establishing a connection between sequences of vectors in the first Hilbert space with a vector in a new Hilbert space whose spectral type with respect to a certain unitary operator is absolutely continuous with respect to the Lebesgue measure. We use this generalization to obtain applications regarding recurrence and multiple ergodic averages when we have measure preserving automorphisms $T$ and $S$ that do not necessarily commute, but $T$ has a maximal spectral type that is mutually singular with the Lebesgue measure.
\end{abstract}

\keywords{van der Corput, recurrence, noncommutative ergodic theory}

\keywords[2020 Mathematics Subject Classification]{\codes[Primary]{47A35, 37A30
}\codes[Secondary]{37B20}}

\end{Frontmatter}

\section{Introduction and statements of results}
\subsection{Introduction} In \cite{OriginalvanderCorput} van der Corput proved Theorem \ref{UniformDistributionvanderCorput'sDifferenceTheorem}, which is now known as van der Corput's Difference Theorem (henceforth abbreviated as vdCDT).

\begin{theorem}[{(\cite[Theorem 1.3.1]{Kuipers&Niederreiter})}]
\label{UniformDistributionvanderCorput'sDifferenceTheorem}
If $(x_n)_{n = 1}^{\infty} \subseteq [0,1]$ is a sequence for which $(x_{n+h}-x_n\pmod{1})_{n = 1}^{\infty}$ is uniformly distributed for all $h \in \mathbb{N}$, then $(x_n)_{n = 1}^{\infty}$ is uniformly distributed.
\end{theorem}

In Ergodic Theory, the following Hilbertian analogues of Theorem \ref{UniformDistributionvanderCorput'sDifferenceTheorem} were introduced by Bergelson in \cite{WMPET} and are of great use.

\begin{theorem}
\label{ClassicalvanderCorput'sDifferenceTheorems}
Let $\mathcal{H}$ be a Hilbert space and $(x_n)_{n = 1}^{\infty} \subseteq \mathcal{H}$ a bounded sequence of vectors. 

\begin{enumerate}
    \item[(i)] If for every $h \in \mathbb{N}$ we have

    \begin{equation}
        \lim_{N\rightarrow\infty}\frac{1}{N}\sum_{n = 1}^N\langle x_{n+h}, x_n\rangle = 0\text{, then }\lim_{N\rightarrow\infty}\left|\left|\frac{1}{N}\sum_{n = 1}^Nx_n\right|\right| = 0.
    \end{equation}
    
    \item[(ii)] If

    \begin{equation}
        \lim_{h\rightarrow\infty}\limsup_{N\rightarrow\infty}\left|\frac{1}{N}\sum_{n = 1}^N\langle x_{n+h}, x_n\rangle\right| = 0\text{, then }\lim_{N\rightarrow\infty}\left|\left|\frac{1}{N}\sum_{n = 1}^Nx_n\right|\right| = 0.
    \end{equation}
    
    \item[(iii)] If
    
    \begin{equation}
        \lim_{H\rightarrow\infty}\frac{1}{H}\sum_{h = 1}^H\limsup_{N\rightarrow\infty}\left|\frac{1}{N}\sum_{n = 1}^N\langle x_{n+h}, x_n\rangle\right| = 0\text{, then }\lim_{N\rightarrow\infty}\left|\left|\frac{1}{N}\sum_{n = 1}^Nx_n\right|\right| = 0.
    \end{equation} 
\end{enumerate}
\end{theorem}

See \cite{ERTAnUpdate}, \cite{PolynomialSzemeredi}, \cite{EinsiedlerWard}, \cite{KbutNotK+1}, and \cite{FurstenbergBook} for some examples of applications of Theorem \ref{ClassicalvanderCorput'sDifferenceTheorems} in Ergodic Theory, and see \cite{vanderCorputTheoremSurvey} for a survey of modern developments regarding vdCDT and its many variations. The main result of this paper is to prove a generalization of Theorem \ref{ClassicalvanderCorput'sDifferenceTheorems}(i) by establishing a connection between Theorem \ref{ClassicalvanderCorput'sDifferenceTheorems}(i) and unitary operators on Hilbert spaces that have Lebesgue spectrum. The following result is a consequence of Lemma \ref{StrongMixingTimesRigidNormAveragesTo0}, Corollary \ref{NearlyStronglyMixingAveragesTo0}, and Theorem \ref{StrongMixingClassicalvdC}.

\begin{theorem}\label{MainResultsInIntro}
Let $\mathcal{H}$ be a Hilbert space and $(x_n)_{n = 1}^{\infty} \subseteq \mathcal{H}$ a bounded sequence of vectors. If for every $h \in \mathbb{N}$ we have

    \begin{equation}
        \lim_{N\rightarrow\infty}\frac{1}{N}\sum_{n = 1}^N\langle x_{n+h}, x_n\rangle = 0,
    \end{equation}
    then $(x_n)_{n = 1}^\infty$ is a spectrally Lebesgue sequence. In particular, if $(c_n)_{n = 1}^\infty \subseteq \mathbb{C}$ is bounded and spectrally singular, then

    \begin{equation}
        \lim_{N\rightarrow\infty}\left|\left|\frac{1}{N}\sum_{n = 1}^Nc_nx_n\right|\right| = 0.
    \end{equation}
    Furthermore, if $\mathcal{H} = L^2(X,\mu)$, and $(y_n)_{n = 1}^\infty \subseteq L^\infty(X,\mu)$ is bounded and spectrally singular, then

    \begin{equation}
        \lim_{N\rightarrow\infty}\left|\left|\frac{1}{N}\sum_{n = 1}^Ny_nx_n\right|\right| = 0.
    \end{equation}
\end{theorem}

Since the discussion necessary to motivate and state the definitions of spectrally Lebesgue and spectrally singular sequences is somewhat long, we defer that discussion to Section \ref{HilbertianvdCSection} and devote the rest of this section to discussing the applications for which Theorems \ref{MainResultsInIntro} and \ref{StrongMixingClassicalvdC} are used, which are Theorems \ref{PartiallyAnswersNFIntro}, \ref{KButNotK+1GeneralizationIntro}, \ref{ApplicationForSingleT}, and \ref{ApplicationForT1T2}. Results that generalize Theorem \ref{UniformDistributionvanderCorput'sDifferenceTheorem} in the same fashion that Theorem \ref{MainResultsInIntro} generalizes Theorem \ref{ClassicalvanderCorput'sDifferenceTheorems}(i) will be the topic of a forthcoming paper by the author. Results that relate Theorems \ref{ClassicalvanderCorput'sDifferenceTheorems}(ii)-(iii) to the ergodic hierarchy of mixing properties of a unitary operator are discussed in Chapters 2.2 and 2.3 of \cite{SohailsPhDThesis}, and analogs of these results in the theory of uniform distribution are discussed in Chapter 2.4 of \cite{SohailsPhDThesis}. Connections between mixing properties and some variations of vdCDT that are not mentioned above were investigated in \cite{UnifiedVDC}. A dynamical proof of vdCDT that is different than what we present in Section \ref{HilbertianvdCSection} can be found in \cite{DynamicalProofOfvdC}.

\subsection{Definitions and notation} Let us now discuss some notation we will be using. If $(X,\mathscr{B},\mu)$ is a probability space, then a measure preserving automorphism $R:X\rightarrow X$ is an invertiable, bi-measurable map for which $\mu(R^{-1}A) = \mu(A) = \mu(RA)$ for all $A \in \mathscr{B}$. We let $\mathcal{I}_R$ be the sub-$\sigma$-algebra of $R$-invariant sets, and $\mathcal{K}_{\text{rat}}(R)$ the sub-$\sigma$-algebra generated by the rational eigenfunctions of $R$. If $\mathscr{A} \subseteq \mathscr{B}$ is a $\sigma$-algebra, then $\mathbb{E}[\cdot|\mathscr{A}]:L^2(X,\mathscr{B},\mu)\rightarrow L^2(X,\mathscr{A},\mu)$ denotes the orthogonal projection. An automorphism $R$ is ergodic if the only $A \in \mathscr{B}$ for which $\mu(A) = \mu(RA)$ satisfy $\mu(A) \in \{0,1\}$, and $R$ is totally ergodic if $R^n$ is ergodic for all $n \in \mathbb{N}$. The automorphism $R$ is weakly mixing if $R\times R$ acting on the probability space $(X\times X,\mathscr{B}\otimes\mathscr{B},\mu^2)$ is ergodic, and $R$ is strongly mixing if $R^n$ converges to $0$ in the weak operator topology. The action of $R$ on $(X,\mathscr{B},\mu)$ naturally lifts to a unitary action on $L^2(X,\mu)$ given (by abuse of notation) by $Rf = f\circ R$. If $f \in L^2(X,\mu)$, then $(\langle R^nf,f\rangle)_{n = 1}^\infty$ is a positive definite sequence, so there is a positive measure $\mu_{f,R}$ on $[0,1]$ for which $\widehat{\mu_{f,R}}(n) = \langle R^nf,f\rangle$, which is called the spectral measure of $f$ with respect to $R$. There exists a measure $\mu_{R}$ such that $\mu_{f,R} << \mu_{R}$, i.e., $\mu_{f,R}$ is absolutely continuous with respect to $\mu_R$, for all $f \in L^2_0(X,\mu)$\footnote{We work with $f \in L^2_0(X,\mu)$ instead of $L^2(X,\mu)$ so that $\mu_{R}$ has a pointmass at $\{0\}$ if and only if $R$ is not ergodic.}, and the maximal spectral type of $R$ is the equivalence class of measures that are mutually absolutely continuous with $\mu_R$. We let $\mathcal{L}$ denote the Lebesgue measure on $[0,1]$, and we say that $R$ has Lebesgue spectrum if the maximal spectral type of $R$ contains $\mathcal{L}$. Similarly, we say that $R$ has singular spectrum if $\mu_R\perp\mathcal{L}$, i.e., if $\mu_R$ is mutually singular with $\mathcal{L}$. When discussing uniform distribution of sequences such as $((n+1)^2\alpha,n^2\alpha)_{n = 1}^\infty$ in $[0,1]^2$ or transformations $R:[0,1]^2\rightarrow[0,1]^2$ such as $R(x,y) = (x+\alpha,y+x)$ for $\alpha \in \mathbb{R}$, it will implicitly be assumed that relevant quantities are reduced modulo $1$. We let $\mathbb{N} = \{1,2,...\}$, and $\delta_x$ denote the measure on $[0,1]$ given by $\delta_x(A) = \mathbbm{1}_A(x)$.

We now define Hardy field functions. Let $B$ denote the set of germs at infinity of real valued functions defined on a half-line $[x,+\infty)$. Then, $(B,+,\cdot)$ is a ring. A sub-field $\mathcal{HF}$ of $B$ that is closed under differentiation is called a Hardy field. An example of a Hardy field is the field $\mathcal{LE}$ of logarithmico-exponential functions. These are the functions defined on some half line of $\mathbb{R}$ by a finite combination of the operations $+,-, \times,\div, \text{exp}, \log$ and composition of functions acting on a real variable
$t$ and real constants. The set $\mathcal{LE}$ contains functions such as the polynomials $p(t)$, $t^c$ for all real $c > 0$, $t+2\log t, t(log t)^2$ and $e^{\sqrt{t}}/t^2$. A function $a:\mathbb{R}_+\rightarrow\mathbb{R}$ is a Hardy field function if $a \in \mathcal{HF}$, for some Hardy field $\mathcal{HF}$ containing $\mathcal{LE}$. The assumption that $\mathcal{HF}\supseteq\mathcal{LE}$ is a necessary in order for us to make use of Theorems \ref{JointErgodicityForHardy} and \ref{WeakMixingHardyPET} later on. We will also assume for convenience that our Hardy fields are translation invariant, i.e., if $a(t) \in \mathcal{HF}$, then $a(t+s) \in \mathcal{HF}$ for any $s \in \mathbb{R}$. We refer the reader to \cite{HardyBackground1,HardyBackground2} and the references therein for more information about Hardy fields.

\subsection{Background} We now recall a result of Frantzikinakis which is related to our applications.

\begin{theorem}[{(\cite[Corollary 1.7]{FSystemsOfHardySequences})}]
\label{FrantzikinakisHardyFieldResult}
Let $a:\mathbb{R}_+\rightarrow\mathbb{R}$ be a Hardy field function for which there exist some $\epsilon > 0$ and $d \in \mathbb{Z}_+$ satisfying

\begin{equation}
    \lim_{t\rightarrow\infty}\frac{a(t)}{t^{d+\epsilon}} = \lim_{t\rightarrow\infty}\frac{t^{d+1}}{a(t)} = \infty.
\end{equation}
Furthermore, let $(X,\mathscr{B},\mu)$ be a probability space and $T,S:X\rightarrow X$ be measure preserving automorphisms, and suppose that $T$ has zero entropy. Then
\begin{enumerate}[(i)]
    \item For every $f,g \in L^\infty(X,\mu)$ we have
    
    \begin{equation}
        \lim_{N\rightarrow\infty}\frac{1}{N}\sum_{n = 1}^NT^nf\cdot S^{\lfloor a(n)\rfloor}g = \mathbb{E}[f|\mathcal{I}_T]\cdot\mathbb{E}[g|\mathcal{I}_S],
    \end{equation}
    where the limit is taken in $L^2(X,\mu)$. 
    
    \item For every $A \in \mathscr{B}$ we have
    
    \begin{equation}
        \lim_{N\rightarrow\infty}\frac{1}{N}\sum_{n = 1}^N\mu\left(A\cap T^{-n}A\cap S^{-\lfloor a(n)\rfloor}A\right) \ge \mu(A)^3.
    \end{equation}
\end{enumerate}
\end{theorem}

To understand the significance of Theorem \ref{FrantzikinakisHardyFieldResult}, we recall that in \cite{LeibmanNilpotentPolynomialSzemeredi} and \cite{WalshErgodicTheorem} it is shown that many multiple ergodic averages of actions of nilpotent groups are well behaved, but some commutativity assumptions are still needed as shown by examples of Furstenberg \cite[Page 40]{FurstenbergBook} and Berend \cite[Example 7.1]{Berend'sExample}. Building upon these examples, the following results were obtained. 

\begin{theorem}[{(\cite[Theorem 1.2]{NoSolvableRoth})}]\label{NoSolvableRecurrence}
    Let $G$ be a finitely generated solvable group of exponential growth.\footnote{A finitely generated solvable group has exponential growth if and only if it contains no nilpotent subgroups of finite index, so the assumption of exponential growth ensures that the group $G$ is not virtually nilpotent. We also recall that Berend's example involved a non-solvable group action.} For any partition $R\bigcup P = \mathbb{Z}\setminus\{0\}$, there exist an action $\{T_g\}_{g \in G}$ of $G$ on a probability space $(X,\mathscr{B},\mu)$, $g_1,g_2 \in G$, and a set $A \in \mathscr{B}$ with $\mu(A) > 0$ such that

    \begin{equation}
        \mu\left(T_{g_1^n}A\cap T_{g_2^n}A\right) = 0\text{ if }n \in R\text{ and }\mu\left(T_{g_1^n}A\cap T_{g_2^n}A\right) \ge \frac{1}{6}\text{ if }n \in P.
    \end{equation}
\end{theorem}

\begin{theorem}[{(\cite[Lemma 4.1]{RandomSequencesAndPtwiseConv})}]\label{DoubleBernoulliIsNutz}
    Let $a,b:\mathbb{N}\rightarrow\mathbb{Z}\setminus\{0\}$ be injective sequences and $F$ be any subset of $\mathbb{N}$. Then there exist a probability space $(X,\mathscr{B},\mu)$, measure preserving automorphisms $T,S:X\rightarrow X$, both of them Bernoulli, and $A \in \mathscr{B}$, such that

    \begin{equation}
        \mu\left(T^{-a(n)}A\cap S^{-b(n)}A\right) = \begin{cases}
                   0&\text{if }n \in F,\\
                   \frac{1}{4}&\text{if }n \notin F.
                \end{cases}
    \end{equation}
\end{theorem}
We can now interpret Theorem \ref{FrantzikinakisHardyFieldResult} as a statement showing that some recurrence results can still be obtained when commutativity assumptions are replaced with structural assumptions (such as zero entropy) and the sequence $a(n)$ produces mixing phenomena that are almost disjoint from the previous structure (cf. \cite[Theorem 1.6]{FSystemsOfHardySequences}). It is worth noting that the assumptions of Theorem \ref{FrantzikinakisHardyFieldResult} require $a(n)$ to be sufficiently far from polynomial sequences. An analogous result for polynomial sequences was conjectured by Frantzikinakis in \cite{FSystemsOfHardySequences}, then obtained by Frantzikinakis and Host.

\begin{theorem}[{(\cite[Page 2]{MultipleRecurrenceAndConvergenceWithoutCommutativity})}]
\label{QuestionOfNF}
Let $(X,\mathscr{B},\mu)$ be a probability space and let $T,S:X\rightarrow X$ be measure preserving autmorphisms. Suppose that $T$ has zero entropy and $f,g \in L^{\infty}(X,\mu)$.
\begin{enumerate}
    \item[(i)] If $p \in \mathbb{Z}[x]$ is a polynomial of degree at least 2, then
    \begin{equation}\label{PolynomialAveragesOfFrantzikinakisConverge}
        \lim_{N\rightarrow\infty}\frac{1}{N}\sum_{n = 1}^NT^nf\cdot S^{p(n)}g
    \end{equation}
    converges in $L^2(X,\mu)$, and is equal to $0$ if either $\mathbb{E}[f|\mathcal{K}_{\text{rat}}(T)] = 0$ or $\mathbb{E}[g|\mathcal{K}_{\text{rat}}(S)] = 0$.
    
    \item[(ii)] For every $A \in \mathscr{B}$ with $\mu(A) > 0$ and $\epsilon > 0$, the set 
    \begin{equation}
        \left\{n \in \mathbb{N}\ |\ \mu\left(A\cap T^{-n}A\cap S^{-p(n)}A\right) > \mu(A)^3-\epsilon\right\}
    \end{equation}
    is syndetic.
\end{enumerate}
\end{theorem}
See also \cite{AZeroEntropyCounterExampleWithoutCommutativity} and \cite{PolynomialFurstenbergJoinings} for further developments in this direction.

\subsection{Applications of the Lebesgue spectrum vdCDT}
In this paper we will obtain many results regarding ergodic averages with non-commuting transformations, but we will be assuming that the automorphism $T$ has singular spectrum instead of having zero entropy. In \cite{SingularSystemsHaveZeroEntropy} it was shown that any $T$ with singular spectrum must also have zero entropy. We defer to Remark \ref{RemarkOnSingularSystems} further discussions regarding systems with singular spectrum so that we can first state our main results.

\begin{theorem}[(cf. Theorem \ref{PartiallyAnswersNFMain})]
\label{PartiallyAnswersNFIntro}
Let $(X,\mathscr{B},\mu)$ be a probability space and let $T,S:X\rightarrow X$ be measure preserving automorphisms for which $T$ has singular spectrum. Let $(k_n)_{n = 1}^{\infty} \subseteq \mathbb{N}$ be a sequence for which $((k_{n+h}-k_n)\alpha)_{n = 1}^{\infty}$ is uniformly distributed in $\overline{\{n\alpha\ |\ n \in \mathbb{N}\}}$ for all $\alpha \in \mathbb{R}$ and $h \in \mathbb{N}$. 
\begin{enumerate}[(i)]
\item For any $f,g \in L^\infty(X,\mu)$ we have

\begin{equation}
    \lim_{N\rightarrow\infty}\frac{1}{N}\sum_{n = 1}^NT^nf\cdot S^{k_n}g = \mathbb{E}[f|\mathcal{I}_T]\mathbb{E}[g|\mathcal{I}_S],
\end{equation}
with convergence taking place in $L^2(X,\mu)$. 

\item If $A\in \mathscr{B}$ then

\begin{equation}
    \lim_{N\rightarrow\infty}\frac{1}{N}\sum_{n = 1}^N\mu\left(A\cap T^{-n}A\cap S^{-k_n}A\right) \ge \mu(A)^3.
\end{equation}
\item If we only assume that $((k_{n+h}-k_n)\alpha)_{n = 1}^{\infty}$ is uniformly distributed for all $\alpha \in \mathbb{R}\setminus\mathbb{Q}$ and $h \in \mathbb{N}$, then (i) and (ii) hold when $S$ is totally ergodic.
\end{enumerate}
\end{theorem}

We see that Theorem \ref{PartiallyAnswersNFIntro}(iii) applies when $k_n = p(n)$ and $p:\mathbb{Z}\rightarrow\mathbb{Z}$ is a polynomial of degree at least $2$, and Theorem \ref{PartiallyAnswersNFIntro}(i)-(ii) when $k_n = \lfloor a(n)\rfloor$ with $a(n)$ as in Theorem \ref{FrantzikinakisHardyFieldResult},\footnote{See \cite[Theorem 2.4.30]{SohailsPhDThesis} for a statement of the uniform distribution properties possessed by such sequences.} which gives a unified proof of Theorems \ref{FrantzikinakisHardyFieldResult} and \ref{QuestionOfNF} when $T$ has singular spectrum.\footnote{While we have not shown that the limits in Equation \eqref{PolynomialAveragesOfFrantzikinakisConverge} exist when $k_n = p(n)$, this is easy to deduce after replacing $g$ with $\mathbb{E}[g|\mathcal{K}_\text{rat}(S)]$. See also Remark \ref{EndOfIntroRemark}.} We may take $k_n = \lfloor a(n)\rfloor$ with $a(n) = n^2\log^2(n)$ as an example satisfying Theorem \ref{PartiallyAnswersNFIntro} and not Theorem \ref{FrantzikinakisHardyFieldResult}, but the remark after Corollary 1.7 of \cite{FSystemsOfHardySequences} shows that a variation of Theorem \ref{FrantzikinakisHardyFieldResult} still applies in this case. This discussion naturally leads us to the following questions.

\begin{question}\label{MyFirstQuestion}
Does there exist a sequence $(k_n)_{n = 1}^\infty \subseteq \mathbb{N}$ for which $((k_{n+h}-k_n)\alpha)_{n = 1}^\infty$ is uniformly distributed in $\overline{\{n\alpha\ |\ n \in \mathbb{N}\}}$ for all $\alpha \in \mathbb{R}$, but for some probability space $(X,\mathscr{B},\mu)$, some measure preserving automorphisms $T,S:X\rightarrow X$ with $(X,\mathscr{B},\mu,T)$ having zero entropy, and some $f,g \in L^\infty(X,\mu)$, we have

\begin{equation}
    \lim_{N\rightarrow\infty}\frac{1}{N}\sum_{n = 1}^NT^nf\cdot S^{k_n}g \neq \mathbb{E}[f|\mathcal{I}_T]\mathbb{E}[g|\mathcal{I}_S]?
\end{equation}
\end{question}

\begin{question}\label{MySecondQuestion}
    Can the conclusions of Theorem \ref{PartiallyAnswersNFIntro}(i)-(ii) be improved if we assume that $(k_{n+h}\alpha,k_n\alpha)_{n = 1}^\infty$ is uniformly distributed in $\overline{\{n\alpha\ |\ n \in \mathbb{N}\}}^2$ for all $\alpha \in \mathbb{R}$ and $h \in \mathbb{N}$? Similarly, can the conclusion of Theorem \ref{PartiallyAnswersNFIntro}(iii) be improved if we assume that $(k_{n+h}\alpha,k_n\alpha)_{n = 1}^\infty$ is uniformly distributed in  $[0,1]^2$ for all $\alpha \in \mathbb{R}\setminus\mathbb{Q}$ and $h \in \mathbb{N}$? 
\end{question}

We defer further discussions regarding Questions \ref{MyFirstQuestion} and \ref{MySecondQuestion} to Remark \ref{RemarkAboutMyQuestions}, at which point the reader will be familiar with the proof of Theorem \ref{PartiallyAnswersNFIntro}. We now recall a theorem of Frantzikinakis, Lesigne, and Wierdl related to our next main result.

\begin{theorem}[{(\cite[Theorem 1.4, Corollary 4.4]{KbutNotK+1})}]
\label{KButNotK+1Thm}
Let $k \ge 2$ be an integer and $\alpha \in \mathbb{R}$ be irrational. Let $R_k = \left\{n \in \mathbb{N}\ |\ n^k\alpha \in \left[\frac{1}{4},\frac{3}{4}\right]\right\}$.
\begin{enumerate}[(i)]
    \item If $(X,\mathscr{B},\mu)$ is a probability space and $S_1,S_2,\cdots,S_{k-1}:X\rightarrow X$ are commuting measure preserving automorphisms, then for any $A \in \mathscr{B}$ with $\mu(A) > 0$, there exists $n \in R$ for which
    
    \begin{equation}
        \mu\left(A\cap S_1^{-n}A\cap S_2^{-n}A\cap\cdots\cap S_{k-1}^{-n}A\right) > 0.
    \end{equation}
    
    \item There exists a measure preserving automorphism $T:X\rightarrow X$ and a set $A \in \mathscr{B}$ satisfying $\mu(A) > 0$ such that for all $n \in R$ we have
    
    \begin{equation}\label{NonRecurrenceEquation}
        \mu\left(A\cap T^{-n}A\cap T^{-2n}A\cap\cdots\cap T^{-kn}A\right) = 0.
    \end{equation}
\end{enumerate}
\end{theorem}
The following result strengthens the conclusion of Theorem \ref{KButNotK+1Thm}(i).

\begin{theorem}[(cf. Theorem \ref{KButNotK+1GeneralizationMain})]
\label{KButNotK+1GeneralizationIntro}
Let $k \ge 2$ be an integer and $\alpha \in \mathbb{R}$ be irrational. Let $R_k = \left\{n \in \mathbb{N}\ |\ n^k\alpha \in \left[\frac{1}{4},\frac{3}{4}\right]\right\}$. Let $(X,\mathscr{B},\mu)$ be a probability space and $S_1,S_2,\cdots,S_{k-1}:X\rightarrow X$ commuting measure preserving automorphisms. Let $T:X\rightarrow X$ be an measure preserving automorphism with singular spectrum, and for which $\{T,S_1,S_2,\cdots,S_{k-1}\}$ generate a nilpotent group. For any $A \in \mathscr{B}$ with $\mu(A) > 0$, there exists $n \in R$ for which
    
    \begin{equation}
        \mu\left(A\cap T^{-n}A\cap S_1^{-n}A\cap S_2^{-n}A\cap\cdots\cap S_{k-1}^{-n}A\right) > 0.
    \end{equation}
\end{theorem}
It is worth noting that in \cite{KbutNotK+1}, for each $k \ge 2$ a uniquely ergodic system $(X_k,T_k)$ with unique invariant measure $\mu$ is constructed such that Equation \eqref{NonRecurrenceEquation} is satisfied for all $n \in R_k$ and some open set $A$. When $k = 2$, we may take $(X_2,T_2) = ([0,1]^2,T)$, where $T(x,y) = (x+\alpha,y+x)$ with $\alpha \in [0,1]\setminus\mathbb{Q}$. Since $T$ has zero topological (and hence measurable) entropy, we see that Theorem \ref{KButNotK+1GeneralizationMain} cannot be extended to the situation in which $T$ has zero (measurable) entropy. It is worth noting that the maximal spectral type of $T$ is $\mathcal{L}+\sum_{n \in \mathbb{Z}\setminus\{0\}}\delta_{n\alpha}$. The fact that we require $\{T,S_1,\cdots,S_{k-1}\}$ to generate a nilpotent group is an artifact of our method of proof, and it is not clear whether or not this assumption can be weakened.

Our next two results both involve the norm convergence of some multiple ergodic averages, but for the sake of simplicity we will only state here special cases of what we prove in Section \ref{SectionOnMeasurePreservingSystems}. In particular we only focus on polynomial functions here and defer further discussion of Hardy field functions to Section \ref{SectionOnMeasurePreservingSystems}. We recall that a polynomial $p \in \mathbb{Q}[x]$ is an \textit{integer polynomial} if for every integer $x$, $p(x)$ is also an integer.

\begin{theorem}[(cf. Theorems \ref{ApplicationForSingleTMain} and \ref{ApplicationWithHardyFunctions})]
\label{ApplicationForSingleT}
Let $(X,\mathscr{B},\mu)$ be a probability space and $T,S:X\rightarrow X$ be measure preserving automorphisms. Suppose that $T$ has singular spectrum and $S$ is totally ergodic. Let $p_1,\cdots,p_K \in \mathbb{Q}[x]$ be integer polynomials for which $\text{deg}(p_1) \ge 2$ and $\text{deg}(p_i) \ge 2+\text{deg}(p_{i-1})$. For any $f,g_1,\cdots,g_K \in L^\infty(X,\mu)$, we have 

\begin{equation}
    \lim_{N\rightarrow\infty}\frac{1}{N}\sum_{n = 1}^NT^nf\prod_{i = 1}^KS^{p_i(n)}g_i = \mathbb{E}[f|\mathcal{I}_T]\prod_{i = 1}^K\int_Xg_id\mu,
\end{equation}
with convergence taking place in $L^2(X,\mu)$.
\end{theorem}
The assumptions made on $p_1,\cdots,p_K$ are an artifact of our method of proof, and it is natural to ask if they can be weakened. Our next result answers this question in the positive, and gives even more, but at the expense of assuming $S$ to be weakly mixing instead of just totally ergodic. We recall that integer polynomials $p$ and $q$ are \textit{essentially distinct} if $p-q$ is not constant.

\begin{theorem}[(cf. Theorems \ref{ApplicationForT1T2Main} and \ref{ApplicationForT1T2Hardy})]
\label{ApplicationForT1T2}
Let $(X,\mathscr{B},\mu)$ be a probability space and $T,R,S:X\rightarrow X$ be measure preserving automorphisms. Suppose that $T$ has singular spectrum, $R$ and $S$ commute, and $S$ is weakly mixing. Let $\ell \in \mathbb{N}$ and let $p_1,\cdots,p_\ell \in \mathbb{Q}[x]$ be pairwise essentially distinct integer polynomials, each having degree at least 2. For any $f,h,g_1,\cdots,g_\ell \in L^\infty(X,\mu)$ satisfying $\int_Xg_jd\mu = 0$ for some $1 \le j \le \ell$, we have

\begin{equation}
    \lim_{N\rightarrow\infty}\frac{1}{N}\sum_{n = 1}^NT^nfR^nh\prod_{j = 1}^{\ell}S^{p_j(n)}g_j = 0,
\end{equation}
with convergence taking place in $L^2(X,\mu)$.
\end{theorem}

We will now give an example to show that the assumptions in Theorems \ref{ApplicationForSingleT} and \ref{ApplicationForT1T2} cannot be weakened significantly. Consider the probability space $([0,1]^2,\mathscr{B},m^2)$ and the measure preserving automorphisms $T(x,y) = (x,y+x)$ for some $\alpha \in \mathbb{R}\setminus\mathbb{Q}$, and $S(x,y) = (x+2\alpha,y+x)$. We see that $T$ and $S$ are both zero entropy automorphisms that are not weakly mixing, and the latter is totally ergodic. Furthermore, we see that $S$ and $T$ generate a 2-step nilpotent group. For $f(x,y) = e^{2\pi i(x-y)}, h(x,y) = e^{2\pi iy},$ and $g(x,y) = e^{-2\pi ix}$, we see that

\begin{align}
    & \lim_{N\rightarrow\infty}\frac{1}{N}\sum_{n = 1}^NT^nf(x,y)S^nh(x,y)S^{\frac{1}{2}(n^2-n)}g(x,y)\\
    =& \lim_{N\rightarrow\infty}\frac{1}{N}\sum_{n = 1}^Ne^{2\pi i \left((1-n)x-y+y+nx+(n^2-n)\alpha-x-(n^2-n)\alpha\right)} = 1 \neq 0.
\end{align}
It is once again worth observing that the maximal spectral type of $T$ is $\mathcal{L}+\delta_0$. Nonetheless, we still state the following conjecture that will be discussed in more detail in Remark \ref{ZeroEntropyConjectureExplanation} after the reader is familiar with the proof of Theorem \ref{ApplicationForT1T2Main}.

\begin{conjecture}\label{ZeroEntropyConjecture}
    Theorem \ref{ApplicationForT1T2} is still true when $T$ has zero entropy.
\end{conjecture}
\begin{remark}\label{RemarkOnSingularSystems}
    In \cite[Proposition 2.9]{RigidityAndNonRecurrence} it is shown that if $(X,\mathscr{B},\mu)$ is a standard probability space, and $\text{Aut}(X,\mathscr{B},\mu)$ is endowed with the strong operator topology, then the set of transformations that are weakly mixing and rigid is a generic set. Since any rigid automorphism has singular spectrum, we see that the set of singular automorphisms is generic. Now let $\mathcal{S} \subseteq \text{Aut}(X,\mathscr{B},\mu)$ denote the collection of strongly mixing transformation, and note that $\mathcal{S}$ is a meager set since an automorphism cannot simultaneously be rigid and strongly mixing. Since $\mathcal{S}$ is not a complete metric space with respect to the topology induced by the strong operator topology, a new topology was introduced in \cite{MetricOnMixing}, with respect to which $\mathcal{S}$ is a complete metric space. It is shown in the Corollary to Theorem 7 of \cite{MetricOnMixing} that a generic $T \in \mathcal{S}$ has singular spectrum, and such a $T$ is mixing of all orders due a well known result of Host \cite{SingularMixingIsMixingOfAllOrders}. See \cite{MixingWithSingularSpectrumFromAFlow} and \cite{CuttingAndStackingMixingAndSingularSpectrum} for concrete examples of $T \in \mathcal{S}$ that have singular spectrum, see \cite{MoreRankOneSingularSpectrum} for more recent developements regarding rank 1 transformations with singular spectrum, see \cite{MildMixingSpectrallDisjointFromStrongMixing} for examples of mildly (but not strongly) mixing systems that have singular spectrum, and see \cite{SubstitutionsWithSingularSpectrum} for examples of substitution systems and interval exchange transformations that have singular spectrum. 
\end{remark}

\begin{remark}\label{EndOfIntroRemark}
Firstly, we remark that we will prove results slightly more general than Theorems \ref{PartiallyAnswersNFIntro}, \ref{KButNotK+1GeneralizationIntro}, \ref{ApplicationForSingleT}, and \ref{ApplicationForT1T2} by showing that we do not need $T$ to have singular spectrum, but only that the $f \in L^2(X,\mu)$ with which we are working satisfies $\mu_{f,T}\perp\mathcal{L}$. Similarly, we will see that in Theorem \ref{PartiallyAnswersNFIntro}(iii) we only need to assume that the spectral measure $\mu_{g,S}$ does not have any rational atoms instead of total ergodicity of $S$. However, it is not clear whether or not Theorems \ref{ApplicationForSingleT} and \ref{ApplicationForT1T2} can be localized in a similar fashion by only requiring the actions of $S$ on $g_i$ to be sufficiently mixing instead of requiring the underlying transformation $S$ to be sufficiently mixing. Secondly, we average over the F\o lner sequence $\{[1,N]\}_{N = 1}^\infty$ solely for the sake of convenience, and our results still hold with the same method of proof when averaging over any other F\o lner sequence $\{[a_n,b_n]\}_{n = 1}^\infty$.\footnote{However, in the case of Theorem \ref{PartiallyAnswersNFIntro} we would have to strengthen the assumption of uniform distribution of $((k_{n+h}-k_n)\alpha)_{n = 1}^\infty$ to the assumption of well distribution.} This means that the main results also hold when taking uniform Ces\`aro averages instead of ordinary Ces\`aro averages.
\end{remark}

\textbf{Acknowledgements:} I would like to thank Srivatsa Srinivas for helpful discussions regarding Fourier analysis that lead to significant improvements in this paper. I would also like to thank the referees for their careful reading of the paper as well as their useful comments which lead to further improvements. I also acknowledge being supported by grant
2019/34/E/ST1/00082 for the project “Set theoretic methods in dynamics and number theory,” NCN (The
National Science Centre of Poland).

\section{Van der Corput's difference theorem and Lebesgue spectrum}\label{HilbertianvdCSection}
\vskip 5mm
Let $\mathcal{H}$ be a Hilbert space. Our desired result, Theorem \ref{StrongMixingClassicalvdC}, can naturally be proven by working with ultrapowers of $\mathcal{H}$, but doing so would require the reader to be familiar ultrafilters. Consequently, in the first half of this section we will discuss how to construct a Hilbert space $\mathscr{H}$ out of sequences of vectors coming from $\mathcal{H}$, which is similar to an ultrapower of $\mathcal{H}$, but can be constructed using elementary arguments. The ideas used in our construction are already present in \cite[Section 3]{TheErdosSumsetPaper}.

Let $||\cdot||$ and $\langle\cdot,\cdot\rangle$ denote the norm and inner product on $\mathcal{H}$ and let $||\cdot||_{\mathscr{H}}$ and $\langle\cdot,\cdot\rangle_{\mathscr{H}}$ denote the norm and inner product on $\mathscr{H}$. We denote the collection of square averageable sequences by
 
 \begin{equation}
     SA(\mathcal{H}) := \left\{(f_n)_{n = 1}^{\infty} \subseteq \mathcal{H}\ |\ \limsup_{N\rightarrow\infty}\frac{1}{N}\sum_{n = 1}^N||f_n||^2 < \infty\right\}.
 \end{equation}
Let $(f_n)_{n = 1}^{\infty}, (g_n)_{n = 1}^{\infty} \in SA(\mathcal{H})$ and observe that

\begin{alignat*}{2}
    &\limsup_{N\rightarrow\infty}\frac{1}{N}\left|\sum_{n = 1}^N\langle f_n, g_n\rangle\right| \le \limsup_{N\rightarrow\infty}\frac{1}{N}\sum_{n = 1}^N||f_n||\cdot||g_n||\numberthis\\
    \le & \left(\limsup_{N\rightarrow\infty}\frac{1}{N}\sum_{n = 1}^N||f_n||^2\right)^{\frac{1}{2}}\left(\limsup_{N\rightarrow\infty}\frac{1}{N}\sum_{n = 1}^N||g_n||^2\right)^{\frac{1}{2}} < \infty.
\end{alignat*}
It follows that we may use diagonalization to construct an increasing sequence of positive integers $(N_q)_{q = 1}^{\infty}$ for which

\begin{equation}
\lim_{q\rightarrow\infty}\frac{1}{N_q}\sum_{n = 1}^{N_q}\langle x_{n+h},y_n\rangle
\end{equation}
exists whenever $(x_n)_{n = 1}^{\infty},(y_n)_{n = 1}^{\infty} \in \left\{(f_n)_{n = 1}^{\infty},(g_n)_{n = 1}^{\infty}\right\}$ and $h \in \mathbb{N}\cup\{0\}$. We now construct a new Hilbert space $\mathscr{H} = \mathscr{H}\left((f_n)_{n = 1}^{\infty},(g_n)_{n = 1}^{\infty},(N_q)_{q = 1}^{\infty}\right)$ from $(f_n)_{n = 1}^{\infty}, (g_n)_{n = 1}^{\infty}$ and $(N_q)_{q = 1}^{\infty}$ as follows. For all $(x_n)_{n = 1}^{\infty},(y_n)_{n = 1}^{\infty} \in \left\{(f_n)_{n = 1}^{\infty},(g_n)_{n = 1}^{\infty}\right\}$ and $h \in \mathbb{N}\cup\{0\}$, we define

\begin{equation}
\left\langle (x_{n+h})_{n = 1}^{\infty},(y_n)_{n = 1}^{\infty}\right\rangle_{\mathscr{H}} := \lim_{q\rightarrow\infty}\frac{1}{N_q}\sum_{n = 1}^{N_q}\langle x_{n+h},y_n\rangle,
\end{equation}
so $\langle\cdot,\cdot\rangle_{\mathscr{H}}$ is a sesquilinear form on $\mathscr{H}' := \text{Span}_{\mathbb{C}}\left(\left\{(f_{n+h})_{n = 1}^{\infty}\right\}_{h = 0}^{\infty}\cup\left\{(g_{n+h})_{n = 1}^{\infty}\right\}_{h = 0}^{\infty}\right)$ with scalar multiplication and addition occuring pointwise. Letting 

\begin{alignat*}{2}
    \mathscr{H}'' := \Biggl\{(e_n)_{n = 1}^{\infty} \in SA(\mathcal{H})\ |\ &\forall\ \epsilon > 0\ \exists\ (h_n(\epsilon))_{n = 1}^{\infty} \in \mathscr{H}'\text{ s.t.}\numberthis\\ &\limsup_{q\rightarrow\infty}\frac{1}{N_q}\sum_{n = 1}^{N_q}||e_n-h_n(\epsilon)||^2 < \epsilon\Biggl\}\text{, and}
\end{alignat*}

\begin{equation}
S = \left\{(x_n)_{n = 1}^{\infty} \in \mathscr{H}''\ |\ \lim_{q\rightarrow\infty}\frac{1}{N_q}\sum_{n = 1}^{N_q}||x_n||^2 = 0\right\},
\end{equation}
we see that $\mathscr{H}''/S$ is a pre-Hilbert space. We will soon see that $\mathscr{H}''$ is sequentially closed under the topology induced by $\langle\cdot,\cdot\rangle_{\mathscr{H}}$, so we define $\mathscr{H}((f_n)_{n = 1}^{\infty},(g_n)_{n = 1}^{\infty},\allowbreak(N_q)_{q = 1}^{\infty}) = \mathscr{H}''/S$. We call $\mathscr{H}((f_n)_{n = 1}^{\infty},(g_n)_{n = 1}^{\infty},\allowbreak(N_q)_{q = 1}^{\infty})$ the Hilbert space induced by $((f_n)_{n = 1}^{\infty},(g_n)_{n = 1}^{\infty},\allowbreak(N_q)_{q = 1}^{\infty})$, and we may write $\mathscr{H}$ in place of $\mathscr{H}((f_n)_{n = 1}^{\infty},(g_n)_{n = 1}^{\infty},\allowbreak(N_q)_{q = 1}^{\infty})$ if $((f_n)_{n = 1}^{\infty},(g_n)_{n = 1}^{\infty},\allowbreak(N_q)_{q = 1}^{\infty})$ is understood from the context.\\

For $(f_n)_{n =1}^{\infty},(g_n)_{n = 1}^{\infty} \in SA(\mathcal{H})$ and $(N_q)_{q = 1}^\infty \subseteq \mathbb{N}$ we say that $((f_n)_{n = 1}^{\infty}, (g_n)_{n = 1}^{\infty},\allowbreak(N_q)_{q = 1}^{\infty})$ is a \textbf{permissible triple} if $\mathscr{H}((f_n)_{n = 1}^{\infty},(g_n)_{n = 1}^{\infty},\allowbreak(N_q)_{q = 1}^{\infty})$ is well defined. Given $(x_n)_{n = 1}^{\infty} \subseteq \mathcal{H}$ for which $(x_n)_{n = 1}^{\infty} \in \mathscr{H}''$, we may view $(x_n)_{n = 1}^{\infty}$ as an element of $\mathscr{H}$ by identifying $(x_n)_{n = 1}^{\infty}$ with its equivalence class in $\mathscr{H}''/S$. We will now show that $\mathscr{H}$ is a Hilbert space by verifying that it is complete.

\begin{theorem}
\label{CompletenessOfMathscrH}
Let $\mathcal{H}$ be a Hilbert space and $(f_n)_{n = 1}^\infty,(g_n)_{n = 1}^\infty \in SA(\mathcal{H})$. Let $ \allowbreak((f_n)_{n = 1}^{\infty},\allowbreak(g_n)_{n = 1}^{\infty},\allowbreak(N_q)_{q = 1}^{\infty})$ be a permissible triple and $\mathscr{H} = \mathscr{H}((f_n)_{n = 1}^{\infty},(g_n)_{n = 1}^{\infty},\allowbreak(N_q)_{q = 1}^{\infty})$. If $\allowbreak\left\{(\xi_{n,m})_{n = 1}^{\infty}\right\}_{m = 1}^{\infty} \subseteq \mathscr{H}''$ is a Cauchy sequence with respect to the metric induced by $||\cdot||_{\mathscr{H}}$, then there exists $(\xi_n)_{n = 1}^{\infty} \in \mathscr{H}''$ for which

\begin{equation}
    \lim_{m\rightarrow\infty}\left(\lim_{q\rightarrow\infty}\frac{1}{N_q}\sum_{n = 1}^{N_q}||\xi_{n,m}-\xi_n||^2\right) = 0.
\end{equation}
In particular, $\mathscr{H}$ is a Hilbert space.
\end{theorem}

\begin{proof} We proceed by modifying the proof of the main result in section \textsection 2 of chapter II of \cite{BohrAndFolner}. Let $(\epsilon_m)_{m = 1}^{\infty}$ be a sequence of real numbers tending to 0 for which

\begin{equation}
    \lim_{q\rightarrow\infty}\frac{1}{N_q}\sum_{n = 1}^{N_q}||\xi_{n,m}-\xi_{n,k}||^2 < \epsilon_m
\end{equation}
whenever $k \ge m$. By induction, let $T_0 = N_0 = 0$ and let $(T_m)_{m = 1}^{\infty} \subseteq \mathbb{N}$ be such that conditions (i)-(iii) below hold.

\begin{itemize}
\item[(i)] For every $m \ge 1$, every $k \ge m$, and every $T \ge T_k$ 

\begin{equation}
    \frac{1}{N_T}\sum_{n = 1}^{N_T}||\xi_{n,k}-\xi_{n,m}||^2 < \epsilon_m.
\end{equation}

\item[(ii)] For every $m \ge 1$ and every $k \ge m$

\begin{equation}
    \frac{1}{N_{T_k}-N_{T_{k-1}}}\sum_{n = N_{T_{k-1}}+1}^{N_{T_k}}||\xi_{n,k}-\xi_{n,m}||^2 < \epsilon_m.
\end{equation}

\item[(iii)] For every $m \ge 1$

\begin{equation}
    \frac{1}{N_{T_m}}\sum_{j = 1}^{m-1}\sum_{n = N_{T_{j-1}}+1}^{N_{T_j}}||\xi_{n,j}-\xi_{n,m}||^2 < \epsilon_m.
\end{equation}
\end{itemize}

Now let us define $(\xi_n)_{n = 1}^{\infty}$ by $\xi_n = \xi_{n,m}$ where $m$ is such that $N_{T_{m-1}} < n \le N_{T_m}$. To conclude the proof, we note that for $m \ge 1$, $k > m$, and $T_{k-1} < T \le T_k$ we have

\begin{alignat*}{2}
    & \sum_{n = 1}^{N_T}||\xi_{n,m}-\xi_n||^2\numberthis\\
    = & \sum_{j = 1}^{m-1}\sum_{n = N_{T_{j-1}}+1}^{N_{T_j}}||\xi_{n,j}-\xi_{n,m}||^2+\sum_{j = m}^{k-1}\sum_{n = N_{T_{j-1}}+1}^{N_{T_j}}||\xi_{n,m}-\xi_n||^2+\sum_{n = N_{T_{k-1}}+1}^{N_T}||\xi_{n,m}-\xi_n||^2\\
    \le & N_{T_m}\epsilon_m+\sum_{j = m+1}^{k-1}(N_{T_j}-N_{T_{j-1}})\epsilon_m+\sum_{n = 1}^{N_T}||\xi_{n,k}-\xi_{n,m}||^2\\
    \le & (N_{T_m}+N_{T_{k-1}}-N_{T_m}+N_T)\epsilon_m \le 2N_T\epsilon_m.
\end{alignat*}
\end{proof}

For our applications in Section \ref{SectionOnMeasurePreservingSystems} we will only be working with uniformly bounded sequences, so we define

\begin{align*}
    &\text{UB}(\mathcal{H}) = \left\{(x_n)_{n = 1}^\infty \in SA(\mathcal{H})\ |\ \lim_{M\rightarrow\infty}\limsup_{N\rightarrow\infty}\frac{1}{N}\sum_{n = 1}^N\left|\left|x_n-(\text{min}(M,||x_n||)\frac{x_n}{||x_n||}\right|\right|^2 = 0\right\},
\end{align*}
which are those sequences in $SA(\mathcal{H})$ that can be well approximated by uniformly bounded sequences. We see by the construction of $\mathscr{H} := \mathscr{H}((f_n)_{n = 1}^\infty,(g_n)_{n = 1}^\infty,(N_q)_{q = 1}^\infty)$ that if $(f_n)_{n = 1}^\infty, (g_n)_{n = 1}^\infty \in \text{UB}(\mathcal{H})$, then $\mathscr{H} \subseteq \text{UB}(\mathcal{H})$. To see why the distinction between $\text{SA}(\mathcal{H})$ and $\text{UB}(\mathcal{H})$ is necessary, we refer the reader to Theorem 2.3.6 and Remark 2.3.7 of \cite{SohailsPhDThesis}.

\begin{definition}\label{DefinitionOfSequencesForETDS}
Let $\mathcal{H}$ be a Hilbert space and $(f_n)_{n = 1}^\infty \in \text{SA}(\mathcal{H})$. 
\begin{enumerate}[(i)]    
    \item A sequence $(f_n)_{n = 1}^\infty$ is \textbf{spectrally singular} if for any permissible triple of the form $P := \left((f_n)_{n = 1}^\infty, (f_n)_{n = 1}^\infty, (N_q)_{q = 1}^\infty\right)$  and $\mathscr{H} = \mathscr{H}(P)$, the sequence $(\langle (f_{n+h})_{n = 1}^\infty, (f_n)_{n = 1}^\infty\rangle_{\mathscr{H}})_{h = 1}^\infty$ is the Fourrier coefficients of a measure $\mu$ such that $\mu\perp\mathcal{L}$.

    \item A sequence $(f_n)_{n = 1}^\infty$ is \textbf{spectrally Lebesgue} if for any permissible triple of the form $P := \left((f_n)_{n = 1}^\infty, (f_n)_{n = 1}^\infty, (N_q)_{q = 1}^\infty\right)$  and $\mathscr{H} = \mathscr{H}(P)$, the sequence $(\langle (f_{n+h})_{n = 1}^\infty, (f_n)_{n = 1}^\infty\rangle_{\mathscr{H}})_{h = 1}^\infty$ is the Fourrier coefficients of a measure $\mu$ such that $\mu<<\mathcal{L}$.
\end{enumerate}
\end{definition}

\begin{remark}\label{RemarkAboutTerminology}
Given a Hilbert space $\mathcal{H}$ and a sequence $(f_n)_{n = 1}^\infty \in \text{SA}(\mathcal{H})$ we may define a left shift operator $L_S$ given by $L_S(f_n)_{n = 1}^\infty = (f_{n+1})_{n = 1}^\infty$, and we note that $L_S$ is naturally identified with a unitary operator on any Hilbert space $\mathscr{H}$ induced by a permissible triple from $\mathcal{H}$. We now see that $\xi := (f_n)_{n = 1}^\infty$ is a spectrally singular sequence if and only if for any permissible triple of the form $\left((f_n)_{n = 1}^\infty,(g_n)_{n = 1}^\infty,\allowbreak(N_q)_{q = 1}^\infty\right)$, we have $\mu_{\xi,L_S}\perp\mathcal{L}$. Similarly, we see that $(f_n)_{n = 1}^\infty$ is a spectrally Lebesgue if and only if we have $\mu_{\xi,L_S}<<\mathcal{L}$, but we point out to the reader that we do not require $\mathcal{L}<<\mu_{\xi,L_S}$.
\end{remark}

\begin{lemma}\label{CreatingWeaklyRigidSequences}
Let $(X,\mathscr{B},\mu,T)$ be a measure preserving system, $f \in L^2(X,\mu)$, $\xi := (T^nf)_{n = 1}^\infty \in \text{UB}(\mathcal{H})$, $P := \left((T^nf)_{n = 1}^\infty,(T^nf)_{n = 1}^\infty,(N_q)_{q = 1}^\infty\right)$ a permissible triple, and $\mathscr{H} = \mathscr{H}(P)$. We have that $\mu_{f,T} = \mu_{\xi,L_S}$.
\end{lemma}

\begin{proof}
It suffices to observe that for any $h \in \mathbb{N}$ we have

\begin{equation}
    \lim_{q\rightarrow\infty}\frac{1}{N_q}\sum_{n = 1}^{N_q}\langle T^{n+h}g, T^ng\rangle = \langle T^hg, g\rangle.
\end{equation}
\end{proof}

\begin{lemma}
\label{LebesgueIsDisjointFromSingular}
Let $\mathcal{H}$ be a Hilbert space and $(f_n)_{n = 1}^\infty, (g_n)_{n = 1}^\infty \in \text{SA}(\mathcal{H})$. If $(f_n)_{n = 1}^\infty$ is spectrally Lebesgue and $(g_n)_{n = 1}^\infty$ is spectrally singular, then

    \begin{equation}
    \lim_{N\rightarrow\infty}\frac{1}{N}\sum_{n = 1}^N\langle f_n, g_n\rangle = 0.
\end{equation}
\end{lemma}

\begin{proof}
Let $(M_q)_{q = 1}^\infty \subseteq \mathbb{N}$ be any increasing sequence for which

\begin{equation}
    \lim_{q\rightarrow\infty}\frac{1}{M_q}\sum_{n = 1}^{M_q}\langle f_n, g_n\rangle
\end{equation}
exists. Let $(N_q)_{q = 1}^\infty$ be a subsequence of $(M_q)_{q = 1}^\infty$ for which $\left((f_n)_{n = 1}^\infty,(g_n)_{n = 1}^\infty,\allowbreak(N_q)_{q = 1}^\infty\right)$ is a permissible triple and observe that

\begin{align}
    &\lim_{q\rightarrow\infty}\frac{1}{M_q}\sum_{n = 1}^{M_q}\langle f_n, g_n\rangle = \lim_{q\rightarrow\infty}\frac{1}{N_q}\sum_{n = 1}^{N_q}\langle f_n, g_n\rangle = \left\langle (f_n)_{n = 1}^\infty,(g_n)_{n = 1}^\infty\right\rangle_{\mathscr{H}} = 0,
\end{align}
where the last equality follows from the fact that $\xi_1 := (f_n)_{n = 1}^\infty$ and $\xi_2 := (g_n)_{n = 1}^\infty$ are vectors in $\mathscr{H}$ with $\mu_{\xi_1,L_S}\perp\mu_{\xi_2,L_S}$.
\end{proof}

\begin{lemma}
\label{StrongMixingTimesRigidNormAveragesTo0}
Let $(X,\mathscr{B},\mu)$ be a probability space, $\mathcal{H} = L^2(X,\mu)$, $(f_n)_{n = 1}^\infty \in \text{SA}(\mathcal{H})$, and $(g_n)_{n = 1}^\infty \in \text{UB}(\mathcal{H})$. If $(f_n)_{n = 1}^\infty$ is spectrally Lebesgue and $(g_n)_{n = 1}^\infty$ is spectrally singular, then

    \begin{equation}
    \lim_{N\rightarrow\infty}\left|\left|\frac{1}{N}\sum_{n = 1}^Nf_ng_n\right|\right| = 0.
\end{equation}
\end{lemma}

\begin{proof}
Let us assume for the sake of contradiction that for some $\epsilon > 0$ and $(N_q)_{q = 1}^\infty \subseteq \mathbb{N}$ we have

\begin{equation}
    \lim_{q\rightarrow\infty}\left|\left|\frac{1}{N_q}\sum_{n = 1}^{N_q}f_ng_n\right|\right| \ge \epsilon.
\end{equation}
By passing to a subsequence of $(N_q)_{q = 1}^\infty$ if necessary, we may assume without loss of generality that
\begin{equation}
    \lim_{q\rightarrow\infty}\left(\left|\left|\frac{1}{N_q}\sum_{n = 1}^{N_{q-1}}f_ng_n\right|\right|+\frac{N_{q-1}}{N_q}\right) = 0.
\end{equation}
For $q \in \mathbb{N}$ let

\begin{equation}
    \xi_q' = \frac{1}{N_{q+1}}\sum_{n = N_q+1}^{N_{q+1}}f_ng_n\text{ and }\xi_q = \frac{\xi_q'}{||\xi_q'||}.
\end{equation}
Now consider the sequence $(G_n)_{n = 1}^\infty \in \text{UB}(\mathcal{H})$ given by $G_n = \xi_q$ for $N_q < n \le N_{q+1}$. Since

\begin{equation}
    \lim_{N\rightarrow\infty}\frac{1}{N}\sum_{n = 1}^N||G_{n+1}-G_n|| \le \lim_{q\rightarrow\infty}\frac{q}{N_q} = 0,
\end{equation}
we see that $(G_n)_{n = 1}^\infty$ is an invariant sequence in the sense that $\left|\left|S(G_n)_{n = 1}^\infty-(G_n)_{n = 1}^\infty\right|\right|_{\mathscr{H}} = 0$ regardless of the permissible triple used to construct $\mathscr{H}$, and we observe that $||G_n|| = 1$ for all $n$. To see that $(\overline{g_n}G_n)_{n = 1}^\infty$ is spectrally singular in $\mathscr{H} = \mathscr{H}((f_n)_{n = 1}^\infty,(\overline{g_n}G_n)_{n = 1}^\infty,(M_q)_{q = 1}^\infty)$ for any valid $(M_q)_{q = 1}^\infty \subseteq (N_q)_{q = 1}^\infty$, it suffices to observe that

\begin{alignat}{2}
&\lim_{q\rightarrow\infty}\frac{1}{M_q}\sum_{n = 1}^{M_q}\left\langle \overline{g_{n+h}}G_{n+h}, \overline{g_n}G_n\right\rangle = \lim_{q\rightarrow\infty}\frac{1}{M_q}\sum_{n = 1}^{M_q}\left\langle \overline{g_{n+h}}G_{n+h}\overline{G_n}, \overline{g_n}\right\rangle\\
=& \lim_{q\rightarrow\infty}\frac{1}{M_q}\sum_{n = 1}^{M_q}\left\langle \overline{g_{n+h}}, \overline{g_n}\right\rangle = \hat{\mu}(h),
\end{alignat}
where $\nu$ is the spectral measure of $(g_n)_{n = 1}^\infty$ in $\mathscr{H}$ with respect to $S$, and $\mu$ is the measure given by $\mu(E) = \nu(1-E)$. To conclude the proof, it suffices to observe that

\begin{alignat*}{2}
    &&\lim_{q\rightarrow\infty}\frac{1}{N_q}\sum_{n = 1}^{N_q}\left\langle f_n, \overline{g_n}G_n\right\rangle = \lim_{q\rightarrow\infty}\frac{1}{N_q}\sum_{n = 1}^{N_q}\left\langle f_ng_n, G_n\right\rangle\numberthis\\ =&& \lim_{q\rightarrow\infty}\frac{1}{N_q}\sum_{n = N_{q-1}+1}^{N_q}\left\langle f_ng_n, \xi_{q-1}\right\rangle = \lim_{q\rightarrow\infty}\left|\left|\xi_{q-1}'\right|\right| \ge \epsilon,
\end{alignat*}
which contradicts Lemma \ref{LebesgueIsDisjointFromSingular}.
\end{proof}

For our next corollary, we recall that $\mathbb{C}$ is a Hilbert space when equipped with the inner product $\langle x,y\rangle_{\mathbb{C}} = x\overline{y}$.

\begin{corollary}
\label{NearlyStronglyMixingAveragesTo0}
If $\mathcal{H}$ is a Hilbert space, $(f_n)_{n = 1}^\infty \in \text{SA}(\mathcal{H})$ is spectrally Lebesgue, and $(c_n)_{n = 1}^\infty \subseteq \mathbb{C}$ is bounded and spectrally singular, then

\begin{equation}
    \lim_{N\rightarrow\infty}\left|\left|\frac{1}{N}\sum_{n = 1}^Nc_nf_n\right|\right| = 0.
\end{equation}
In particular, we have that

\begin{equation}
    \lim_{N\rightarrow\infty}\left|\left|\frac{1}{N}\sum_{n = 1}^Nf_n\right|\right| = 0.
\end{equation}
\end{corollary}

\begin{proof}
    Firstly, we observe that the smallest subspace of $\mathcal{H}$ containing $(f_n)_{n = 1}^\infty$ is separable, so it is unitarily isomorphic to a subspace of $L^2([0,1],\mathcal{L})$, hence we may assume without loss of generality that $(f_n)_{n = 1}^\infty \in \text{SA}(L^2([0,1],\mathcal{L})$. Next, let $(g_n)_{n = 1}^\infty \in \text{SA}(L^2(X,\mathcal{L}))$ given by $g_n = c_n\mathbbm{1}_{[0,1]}$, and observe that for any permissible triple $P = ((c_n)_{n = 1}^\infty,(c_n)_{n = 1}^\infty,(N_q)_{q = 1}^\infty)$ and $h \in \mathbb{N}$ we have

    \begin{equation}
        \lim_{N\rightarrow\infty}\frac{1}{N}\sum_{n = 1}^N\langle g_{n+h},g_n\rangle = \lim_{N\rightarrow\infty}\frac{1}{N}\sum_{n = 1}^Nc_{n+h}\overline{c_n},
    \end{equation}
    so $(g_n)_{n = 1}^\infty$ is spectrally singular. We may now use Lemma \ref{StrongMixingTimesRigidNormAveragesTo0} to see that 

    \begin{equation}
        0 = \lim_{N\rightarrow\infty}\left|\left|\frac{1}{N}\sum_{n = 1}^Nf_ng_n\right|\right| = \lim_{N\rightarrow\infty}\left|\left|\frac{1}{N}\sum_{n = 1}^Nc_nf_n\right|\right|.
    \end{equation}
    For the latter half of the theorem, it suffices to observe that the sequence $(c_n)_{n = 1}^\infty$ given by $c_n = 1$ for all $n$ will always have $\delta_0$ as its spectral measure, and is consequently a spectrally singular sequence.
\end{proof}

This shows us that the following theorem, which is the main result of this section, is indeed a generalization of Theorem \ref{ClassicalvanderCorput'sDifferenceTheorems}(i).

\begin{theorem}
\label{StrongMixingClassicalvdC}
Let $\mathcal{H}$ be a Hilbert space and $(f_n)_{n = 1}^\infty \in \text{SA}(\mathcal{H})$. If

\begin{equation}\label{EquationForMainApplication}
    \sum_{h = 1}^\infty\limsup_{N\rightarrow\infty}\left|\frac{1}{N}\sum_{n = 1}^N\langle f_{n+h},f_n\rangle\right|^2 < \infty
\end{equation}
for all $h \in \mathbb{N}$, then $(f_n)_{n = 1}^\infty$ is spectrally Lebesgue.
\end{theorem}

\begin{proof}
    Let $P := \left((f_n)_{n = 1}^\infty,(f_n)_{n = 1}^\infty,(N_q)_{q = 1}^\infty\right)$ be a permissible triple and $\mathscr{H} = \mathscr{H}(P)$. Let $\mu$ be a positive measure on $[0,1]$ for which
    
    \begin{equation}
        \hat{\mu}(h) = \lim_{q\rightarrow\infty}\frac{1}{N_q}\sum_{n = 1}^{N_q}\langle f_{n+h}, f_n\rangle = \left\langle (f_{n+h})_{n = 1}^\infty. (f_n)_{n = 1}^\infty\right\rangle_{\mathscr{H}}.
    \end{equation}
    Since the fourier coefficients of $\mu$ are square summable, we may define $g \in L^2([0,1],\mathcal{L})$ by

    \begin{equation}
        g(x) = \sum_{h \in \mathbb{Z}}\hat{\mu}(h)e^{2\pi ihx},
    \end{equation}
    and we see that $\int_0^1e^{-2\pi ihx}gd\mathcal{L}(x) = \hat{\mu}(h) = \int_0^1e^{-2\pi ihx}d\mu(x)$ for all $h \in \mathbb{Z}$. Since a measure is uniquely determined by its Fourier coefficients, we see that $d\mu = gd\mathcal{L}$.
\end{proof}

\begin{remark}
    It is worth mentioning that a sequence $(c_n)_{n = 1}^\infty \subseteq \mathbb{C}$ being square summable is a sufficient but not necessary condition for the existince of a measure $\mu<<\mathcal{L}$ for which $\hat{\mu}(n) = c_n$. It follows that Theorem \ref{StrongMixingClassicalvdC} could, in principle, be generalized by trying to weaken the condition of Equation \eqref{EquationForMainApplication}. However, it is unlikely that condition of Equation \eqref{EquationForMainApplication} can be weakened in an aesthetic manner. To see why, we begin by observing that the main simplification is the use of limit supremums to avoid the technicalities of permissible triples. To see that a similar simplification cannot be made in general, we will show that the property of being the Fourier coefficients of a measure $\mu<<\mathcal{L}$ depends on more than just the magnitude of the coefficients. Consider a strongly mixing measure preserving system $(X,\mathscr{B},\mu,T)$ that has singular spectrum, such as the one constructed in \cite{MixingWithSingularSpectrumFromAFlow}. Let $f \in L^2(X,\mu)$ be such that $\int_Xfd\mu = 0$, and observe that $\lim_{n\rightarrow\infty}\widehat{\mu_{f,T}}(n) = 0$. We can construct a convex sequence of real numbers $c_n$ for which $c_n \ge |\widehat{\mu_{f,t}}(n)|$ and $\lim_{n\rightarrow\infty}c_n = 0$. It is a classical result \cite[Theorem 4.1]{KatznelsonHarmonicAnalysis} that there exists a measure $\nu<<\mathcal{L}$ for which $\hat{\nu}(n) = c_n$.
\end{remark}

While our next result will not be used in this paper, we include it here so that Theorem \ref{StrongMixingClassicalvdC} can be compared to the generalizations of vdCDT found in \cite[Chapter 2.2]{SohailsPhDThesis}.

\begin{lemma}
    Let $\mathcal{H}$ be a Hilbert space and $(f_n)_{n = 1}^\infty,(g_n)_{n = 1}^\infty \in \text{SA}(\mathcal{H})$ be such that $(f_n)_{n = 1}^\infty$ is spectrally Lebesgue. If $P := \left((f_n)_{n = 1}^\infty,(g_n)_{n = 1}^\infty,(N_q)_{q = 1}^\infty\right)$ is a permissible triple, then there exists a complex-valued measure $\mu$ such that $\mu<<\mathcal{L}$ and

     \begin{equation}
         \lim_{q\rightarrow\infty}\frac{1}{N_q}\sum_{n = 1}^{N_q}\langle f_{n+h}, g_n\rangle = \hat{\mu}(h),
     \end{equation}
     for all $h \in \mathbb{N}$.
\end{lemma}

\begin{proof}
    Let $\mathscr{H} = \mathscr{H}(P)$, $\xi_f = (f_n)_{n = 1}^\infty$, $\xi_g = (g_n)_{n = 1}^\infty$, $\mathscr{H}_f := c\ell(\text{Span}_{\mathbb{C}}(\{L_S^n\xi_f\ |\ n \in \mathbb{Z}\}))$, and $P:\mathscr{H}\rightarrow\mathscr{H}_f$ the orthgonal projection. Using a standard variation of the polarization identity, we see that for all $h \in \mathbb{N}$ we have

    \begin{alignat*}{2}
        4\langle L_S^h\xi_f,P\xi_g\rangle_{\mathscr{H}} = & \langle L_S^h\xi_f+L_S^hP\xi_g,\xi_f+P\xi_g\rangle_{\mathscr{H}}-\langle L_S^h\xi_f-L_S^hP\xi_g,\xi_f-P\xi_g\rangle_{\mathscr{H}}\\
        &+i\langle L_S^h\xi_f+iL_S^hP\xi_g,\xi_f+iP\xi_g\rangle_{\mathscr{H}}-i\langle L_S^h\xi_f-iL_S^hP\xi_g,\xi_f-iP\xi_g\rangle_{\mathscr{H}}.
    \end{alignat*}
    Since $(\langle L_S^h\xi_f+aL_S^hP\xi_g,\xi_f+aP\xi_g\rangle_{\mathscr{H}})_{h = 1}^\infty$ is the fourier coefiecients of the spectral measure $\mu_a := \mu_{\xi_f+aP\xi_g,L_s}$, and $\xi_f+aP\xi_g \in \mathscr{H}_f$, we see that $\mu_a << \mu_{\xi_f,L_S} << \mathcal{L}$. To conclude, we observe that 
    
    \begin{equation}
        \langle L_S^h\xi_f,\xi_g\rangle_{\mathscr{H}} = \langle L_S^h\xi_f,P\xi_g\rangle_{\mathscr{H}}\text{, hence }\mu = \frac{1}{4}(\mu_1-\mu_{-1}+i\mu_i-i\mu_{-i}).
    \end{equation}
\end{proof}
\section{Applications to recurrence and ergodic averages in measure preserving systems}\label{SectionOnMeasurePreservingSystems}
\subsection{Preliminary Results}
In this subsection we collect some known results from the literature that will be used later on. All limits in this subsection converge in $L^2(X,\mu)$.

\begin{theorem}[{(Bergelson \cite[Theorem 1.2]{WMPET})}]
    Suppose $(X,\mathscr{B},\mu,S)$ is a weakly mixing system and $p_1,p_2,\cdots,p_k \in \mathbb{Q}[x]$ are pairwise essentially distinct integer polynomials. Then for any $g_1,g_2,\cdots,g_k \in L^\infty(X,\mu)$, we have

    \begin{equation}
    \lim_{N\rightarrow\infty}\frac{1}{N}\sum_{n = 1}^N\prod_{i = 1}^kS^{p_i(n)}g_i = \prod_{i = 1}^K\int_Xg_id\mu.
\end{equation}
\end{theorem}

We recall that a collection of integer polynomials $p_1,\cdots,p_k \in \mathbb{Q}[x]$ is \textit{independent} if any nontrivial rational linear combination of them is not constant. 

\begin{theorem}[{(Frantzikinakis, Kra \cite[Theorem 1.1]{PolynomialAveragesConverge})}]\label{PolynomialAveragesConvergeForSingleT}
Let $(X,\mathscr{B},\mu,S)$ be a totally ergodic system, and let $p_1,\cdots,p_k \in \mathbb{Q}[x]$ be a collection of independent integer polynomials. Then for $g_1,\cdots,g_K \in L^\infty(X,\mu)$, we have

\begin{equation}
    \lim_{N\rightarrow\infty}\frac{1}{N}\sum_{n = 1}^N\prod_{i = 1}^kS^{p_i(n)}g_i = \prod_{i = 1}^K\int_Xg_id\mu.
\end{equation}
\end{theorem}

A function $a \in \mathcal{HF}$ has \textit{polynomial growth} if there exists a $d \in \mathbb{N}$ for which $\lim_{t\rightarrow\infty}a(t)t^{-d} = 0$.

\begin{theorem}[{(Tsinas \cite[Theorem 1.2]{JointErgodicityForHardyFieldSequences})}]\label{JointErgodicityForHardy}
     Let $a_1,\cdots,a_k \in \mathcal{HF}$ have polynomial growth and assume that if $a$ is a nontrivial $\mathbb{R}$-linear combination of $\{a_i\}_{i = 1}^k$, then

     \begin{equation}
         \lim_{t\rightarrow\infty}\frac{\left|a(t)-\alpha p(t)\right|}{\log(t)} = \infty\text{ for any }p(t) \in \mathbb{Z}[t], \alpha \in \mathbb{R}. 
     \end{equation}
     Then for any ergodic m.p.s. $(X,\mathscr{B},\mu,T)$ and functions $f_1,\cdots,f_k \in L^\infty(X,\mu)$, we have

     \begin{equation}
        \lim_{N\rightarrow\infty}\frac{1}{N}\sum_{n = 1}^N\prod_{i = 1}^kT^{\lfloor a_i(n)\rfloor}g_i = \prod_{i = 1}^k\int_Xg_id\mu.
    \end{equation}
\end{theorem}

\begin{theorem}[{(Tsinas \cite[Corollary 1.4]{JointErgodicityForHardyFieldSequences})}]\label{WeakMixingHardyPET}
    Assume that the functions $a_1,\cdots,a_k \in \mathcal{HF}$ have polynomial growth and satisfy

    \begin{alignat*}{2}
        &\lim_{t\rightarrow\infty}\frac{|a_i(t)|}{\log(t)} = \infty\text{ for all }1 \le i \le k\text{, and}\\
        &\lim_{t\rightarrow\infty}\frac{|a_i(t)-a_j(t)|}{\log(t)} = \infty\text{ for all }1 \le i < j \le k.
    \end{alignat*}
    Then for any weakly mixing m.p.s. $(X,\mathscr{B},\mu,T)$, we have

    \begin{equation}
        \lim_{N\rightarrow\infty}\frac{1}{N}\sum_{n = 1}^N\prod_{i = 1}^kT^{\lfloor a_i(n)\rfloor}g_i = \prod_{i = 1}^k\int_Xg_id\mu.
    \end{equation}
\end{theorem}
\subsection{Main Results}
\begin{theorem}
\label{PartiallyAnswersNFMain}
Let $(X,\mathscr{B},\mu)$ be a probability space and let $T,S:X\rightarrow X$ be measure preserving automorphisms. Let $(k_n)_{n = 1}^{\infty} \subseteq \mathbb{N}$ be a sequence for which $((k_{n+h}-k_n)\alpha)_{n = 1}^{\infty}$ is uniformly distributed in the orbit closure of $\alpha$ for all $\alpha \in \mathbb{R}$ and $h \in \mathbb{N}$. 
\begin{enumerate}[(i)]
\item For any $f,g \in L^\infty(X,\mu)$ and $\mu_{f,T}\perp\mathcal{L}$, then

\begin{equation}
    \lim_{N\rightarrow\infty}\frac{1}{N}\sum_{n = 1}^NT^nf\cdot S^{k_n}g = \mathbb{E}[f|\mathcal{I}_T]\mathbb{E}[g|\mathcal{I}_S],
\end{equation}
with convergence taking place in $L^2(X,\mu)$. 

\item If $A\in \mathscr{B}$ $\mu_{\mathbbm{1}_A,T}\perp\mathcal{L}$ then

\begin{equation}
    \lim_{N\rightarrow\infty}\frac{1}{N}\sum_{n = 1}^N\mu\left(A\cap T^{-n}A\cap S^{-k_n}A\right) \ge \mu(A)^3.
\end{equation}
\item If we only assume that $((k_{n+h}-k_n)\alpha)_{n = 1}^{\infty}$ is uniformly distributed for all $\alpha \in \mathbb{R}\setminus\mathbb{Q}$ and $h \in \mathbb{N}$, then (i) holds when we further assume that $\mathbb{E}[g|\mathcal{K}_{\text{rat}}(S)] = \mathbb{E}[g|\mathcal{I}_S]$ and (ii) holds when we further assume that $\mathbb{E}[\mathbbm{1}_A|\mathcal{K}_{\text{rat}}(S)] = \mathbb{E}\left[\mathbbm{1}_A|\mathcal{I}_S\right]$.
\end{enumerate}
\end{theorem}

For the proofs of (i)-(iii) all limits of sequences of vectors in $L^2(X,\mu)$ will be with respect to norm convergence in $L^2(X,\mu)$.

\begin{proof}[Proof of (i)]
Let $g' = g-\mathbb{E}[g|\mathcal{I}_S]$. We will use Theorem \ref{StrongMixingClassicalvdC} to show that $(S^{k_n}g')_{n = 1}^{\infty}$ is spectrally Lebesgue. Recall that $\widehat{\nu_{g',S}}(n) = \langle S^ng',g'\rangle$ for all $n \in \mathbb{N}$, and observe that

\begin{alignat*}{2}
\label{BergelsonianHerglotzEquation}
    &\lim_{N\rightarrow\infty}\frac{1}{N}\sum_{n = 1}^N\langle S^{k_{n+h}}g',S^{k_n}g'\rangle = \lim_{N\rightarrow\infty}\frac{1}{N}\sum_{n = 1}^N\langle S^{k_{n+h}-k_n}g',g'\rangle\numberthis\\
    = & \lim_{N\rightarrow\infty}\frac{1}{N}\sum_{n = 1}^N\widehat{\nu_{g',S}}\left(k_{n+h}-k_n\right) = \lim_{N\rightarrow\infty}\frac{1}{N}\sum_{n = 1}^N\int_0^1e^{2\pi i(k_{n+h}-k_n)x}d\nu_{g',S}(x)\\
    = & \int_0^1\lim_{N\rightarrow\infty}\frac{1}{N}\sum_{n = 1}^Ne^{2\pi i(k_{n+h}-k_n)x}d\nu(x) = 0,
\end{alignat*}
where the last equality follows from the fact that $(\left(k_{n+h}-k_n\right)\alpha)_{n = 1}^{\infty}$ is uniformly distributed in the orbit closure of $\alpha$ for all $\alpha\in [0,1]$, and $\nu_{g',S}(\{0\}) = ||\mathbb{E}[g'|\mathcal{I}_S]||^2 = 0$.\\

Since $\mu_{f,T}\perp\mathcal{L}$, $(T^nf)_{n = 1}^{\infty} \in \text{UB}(L^2(X,\mu))$ is spectrally singular by Lemma \ref{CreatingWeaklyRigidSequences}. We may now use Lemma \ref{StrongMixingTimesRigidNormAveragesTo0} to see that

\begin{equation}
    \lim_{N\rightarrow\infty}\frac{1}{N}\sum_{n = 1}^NT^nf\cdot S^{k_n}g' = 0.
\end{equation}
Lastly, we use the mean ergodic theorem to see that

\begin{equation}
    \lim_{N\rightarrow\infty}\frac{1}{N}\sum_{n = 1}^NT^nf\cdot \mathbb{E}[g|\mathcal{I}_S] = \mathbb{E}[f|\mathcal{I}_T]\cdot\mathbb{E}[g|\mathcal{I}_S].
\end{equation}
\end{proof}

\begin{proof}[Proof of (ii)]
 Using \cite[Lemma 1.6]{MRf2CT} we see that for any bounded nonnegative $h \in L^2(X,\mu)$ we have

\begin{equation}
    \int_Xh\cdot\mathbb{E}[h|\mathcal{I}_T]\cdot\mathbb{E}[h|\mathcal{I}_S]d\mu \ge \left(\int_Xhd\mu\right)^3,
\end{equation}
so the desired result follows from part (i) after setting $h = f = g = \mathbbm{1}_A$.
\end{proof}

\begin{proof}[Proof of (iii)]
Since part (ii) was proven as a result of part (i), it suffices to only show that part (i) still holds in this new situation. To this end, it suffices to repeat the proof of (i) and observe that the measure $\mu_{g',S}$ now satisfies $\mu_{g',S}(\mathbb{Q}\cap[0,1]) = ||\mathbb{E}[g'|\mathcal{K}_{\text{rat}}(S)]||^2 = 0$ instead of just $\mu_{g',S}(\{0\}) = 0$. Consequently, we see that the last equation of \eqref{BergelsonianHerglotzEquation} will still hold since for all $x \in [0,1]\setminus\mathbb{Q}$ we have

\begin{equation}
    \lim_{N\rightarrow\infty}\frac{1}{N}\sum_{n = 1}^Ne^{2\pi i(k_{n+h}-k_n)x} = 0.
\end{equation}
\end{proof}

\begin{remark}\label{RemarkAboutMyQuestions}
    To see that our methods cannot be used to handle the more general case in which $T$ has zero entropy, we consider the horocycle flow $(X,\mathscr{B},\mu,T)$. For our discussion we don't require a concrete definition of the horocycle flow, but let us recall some of its dynamical properties. The horocycle flow is a minimal \cite{MinimalityOfTheHorocycleFlow} uniquely ergodic \cite{UniqueErgodicityOfTheHorocycleFlow} dynamical system with Lebesgue spectrum \cite{HorocycleFlowHasLebesgueSpectrum} that is mixing of all orders \cite{HorocycleFlowIsMixingOfAllOrders} and has zero entropy \cite{HorocycleFlowHasZeroEntropy}. It follows that if $0 \neq g \in L^2_0(X,\mu)$, then $\mu_{\overline{g},T}<<\mathcal{L}$, so $(T^n\overline{g})_{n = 1}^\infty$ will be spectrally Lebesgue as a result of Lemma \ref{CreatingWeaklyRigidSequences}, but we clearly have that
    
    \begin{equation}
        \lim_{N\rightarrow\infty}\frac{1}{N}\sum_{n = 1}^NT^ng\cdot T^n\overline{g} = \int_X|g|^2d\mu \neq 0 = \left(\int_Xgd\mu\right)^2,
    \end{equation}
    with convergence taking place in $L^2(X,\mu)$. This example shows that for a sequence of vectors $(g_n)_{n = 1}^\infty$ to have no correlation with a sequence generated by a zero entropy system, being spectrally Lebesgue is not sufficient. Since the assumptions on $(k_n)_{n = 1}^\infty$ in Theorem \ref{PartiallyAnswersNFMain} seem to be the smallest assumptions needed to make $(T^{k_n}g)_{n = 1}^\infty$ spectrally Lebesgue, we believe that Question \ref{MyFirstQuestion} should have a positive answer.
    
     The motivation for Question \ref{MySecondQuestion} comes from the observations that if $p(x) \in \mathbb{Z}[x]$ has degree at least $2$, $(p(n+h)\alpha,p(n)\alpha)_{n = 1}^\infty$ is uniformly distributed in $[0,1]^2$ for all $h \in \mathbb{N}$ and $\alpha \in \mathbb{R}\setminus\mathbb{Q}$, and if $a(n)$ is as in Theorem \ref{FrantzikinakisHardyFieldResult} then $(\lfloor a(n+h)\rfloor\alpha,\lfloor a(n)\rfloor\alpha)_{n = 1}^\infty$ is uniformly distributed in the square of the orbit closure of $\alpha$ for all $h \in \mathbb{N}$ and $\alpha \in \mathbb{R}$.\footnotemark[2] While it would be convenient for the conditions of Question \ref{MySecondQuestion} to be sufficient for generalizing Theorem \ref{PartiallyAnswersNFMain} to the case in which $T$ has zero entropy transformation, we are currently only able to show that these conditions are not necessary. Let us consider the sequence given by $x_n = n^2$ if $n$ is odd and $x_n = 2(n-1)^2$ if $n$ is even. It is shown in \cite[Theorems 2.4.26 and 2.4.28]{SohailsPhDThesis} that for $\alpha \in \mathbb{R}\setminus\mathbb{Q}$, $((x_{n+h}-x_n)\alpha))_{n = 1}^\infty$ is uniformly distributed for all $h \in \mathbb{N}$, but we see that $(x_n\alpha,x_{n+1}\alpha)_{n = 1}^\infty$ is not uniformly distributed since $x_{n+1} = 2x_n$ when $n$ is odd. Nonetheless, we see that Theorem \ref{QuestionOfNF} still holds if $p(n)$ is replaced by $x_n$, because if $T$ has zero entropy then $T^2$ does as well, and $x_n$ decomposes into polynomial sequences along the even and odd values of $n$. 
\end{remark}

\begin{theorem}
\label{KButNotK+1GeneralizationMain}
Let $k \ge 2$ be an integer and $\alpha \in \mathbb{R}$ be irrational. Let $R_k = \left\{n \in \mathbb{N}\ |\ n^k\alpha \in \left[\frac{1}{4},\frac{3}{4}\right]\right\}$. Let $(X,\mathscr{B},\mu)$ be a probability space and $S_1,S_2,\cdots,S_{k-1}:X\rightarrow X$ commuting measure preserving automorphisms. Let $T:X\rightarrow X$ be an measure preserving automorphism for which $\{T,S_1,S_2,\cdots,S_{k-1}\}$ generate a nilpotent group. For any $A \in \mathscr{B}$ with $\mu(A) > 0$ and $\mu_{\mathbbm{1}_A,T}\perp\mathcal{L}$, there exists $n \in R$ for which
    
    \begin{equation}
        \mu\left(A\cap T^{-n}A\cap S_1^{-n}A\cap S_2^{-n}A\cap\cdots\cap S_{k-1}^{-n}A\right) > 0.
    \end{equation}
\end{theorem}

\begin{proof}
We remark that our proof is essentially the same as that in \cite[Corollary 4.4]{KbutNotK+1} other than the fact that we use Theorem \ref{StrongMixingClassicalvdC} in place of Theorem \ref{ClassicalvanderCorput'sDifferenceTheorems}(i), and that we use nilpotent ergodic theorems of Walsh and Leibman in place of the Furstenberg-Katznelson multiple recurrence theorem. We begin with a useful proposition.

\begin{prop}
\label{PropCopiedFromFLW}
Let $k \in \mathbb{N}$ be at least 2, $(X,\mathscr{B},\mu)$ be a probability space, and $S_1,\cdots,S_{k-1}:X\rightarrow X$ be commuting measure preserving automorphisms. Let $p(x) \in \mathbb{R}[x]$ have degree at least $k$ and an irrational leading coefficient. For any $f_1,\cdots,f_{k-1} \in L^{\infty}(X,\mu)$ the sequence

\begin{equation}
\label{ANearlyOrthogonalSequenceApplication}
    \left(S_1^nf_1S_2^nf_2\cdots S_{k-1}^nf_{k-1}e^{2\pi ip(n)}\right)_{n = 1}^\infty
\end{equation}
is spectrally Lebesgue.
\end{prop}

\begin{proof}[Proof of Proposition \ref{PropCopiedFromFLW}]
We proceed by induction on $k$. Let $a_n = S_1^nf_1S_2^nf_2\cdots S_{k-1}^nf_ke^{2\pi ip(n)}$, $g_{i,h} = S_i^hf_i\overline{f_i}$, $\tilde{S}_i = S_iS_1^{-1}$, and observe that for $h \in \mathbb{N}$ we have

\begin{alignat*}{2}
&\lim_{N\rightarrow\infty}\frac{1}{N}\sum_{n = 1}^N\langle a_{n+h}, a_n\rangle\numberthis\\
= & \lim_{N\rightarrow\infty}\frac{1}{N}\sum_{n = 1}^N\int_XS_1^ng_{1,h}S_2^ng_{2,h}\cdots S_{k-1}^ng_{k-1,h}e^{2\pi i(p(n+h)-p(n))}d\mu\\
= & \lim_{N\rightarrow\infty}\frac{1}{N}\sum_{n = 1}^N\int_Xg_{1,h}\tilde{S}_2^ng_{2,h}\cdots \tilde{S}_{k-1}^ng_{k-1,h}e^{2\pi i(p(n+h)-p(n))}d\mu = 0,
\end{alignat*}
where the last equality follows from the uniform distribution of $(p(n+h)-p(n))_{n = 1}^\infty\pmod{1}$ when $k = 2$, and from the inductive hypothesis and Corollary \ref{NearlyStronglyMixingAveragesTo0} when $k > 2$. The fact that $(a_n)_{n = 1}^\infty$ is spectrally Lebesgue now follows from Theorem \ref{StrongMixingClassicalvdC}.
\end{proof}

Returning to the proof of Theorem \ref{KButNotK+1GeneralizationMain}, we use Lemma \ref{CreatingWeaklyRigidSequences} to see that $(T^n\mathbbm{1}_A)_{n = 1}^\infty$ is spectrally singular, and we also observe that the sequence in equation \eqref{ANearlyOrthogonalSequenceApplication} with $f_i = \mathbbm{1}_A$ for all $i$ is spectrally Lebesgue. An application of Lemma \ref{StrongMixingTimesRigidNormAveragesTo0} shows us that for any $m \in \mathbb{Z}\setminus\{0\}$ we have

\begin{alignat*}{2}
    &\lim_{N\rightarrow\infty}\left|\left|\frac{1}{N}\sum_{n = 1}^Ne^{2\pi imn^k\alpha}S_1^n\mathbbm{1}_AS_2^n\mathbbm{1}_A\cdots S_{k-1}^n\mathbbm{1}_AT^n\mathbbm{1}_A\right|\right| = 0\text{, hence}\\
    &\lim_{N\rightarrow\infty}\frac{1}{N}\sum_{n = 1}^N\mu\left(A\cap T^{-n}A\cap S_1^{-n}A\cap S_2^{-n}A\cap\cdots\cap S_{k-1}^nA\right)e^{2\pi imn^k\alpha}\numberthis\\
    =& \lim_{N\rightarrow\infty}\frac{1}{N}\sum_{n = 1}^N\left\langle \mathbbm{1}_A, e^{2\pi imn^k\alpha}S_1^n\mathbbm{1}_AS_2^n\mathbbm{1}_A\cdots S_{k-1}^n\mathbbm{1}_AT^n\mathbbm{1}_A\right\rangle = 0.
\end{alignat*}
Using standard linearity arguments, we now see that

\begin{alignat*}{2}
&\lim_{N\rightarrow\infty}\frac{1}{N}\sum_{n = 1}^N\mathbbm{1}_{[\frac{1}{4},\frac{3}{4}]}(n^k\alpha)\mu\left(A\cap T^{-n}A\cap S_1^{-n}A\cap S_2^{-n}A\cap\cdots\cap S_{k-1}^nA\right)\numberthis\\
= &\lim_{N\rightarrow\infty}\frac{1}{N}\sum_{n = 1}^N\frac{1}{2}\mu\left(A\cap T^{-n}A\cap S_1^{-n}A\cap S_2^{-n}A\cap\cdots\cap S_{k-1}^nA\right) > 0,
\end{alignat*}
where the existence of the last limit follows from \cite[Theorem 1.1]{WalshErgodicTheorem}, and the last inequality follows from \cite[Theorem NM]{LeibmanNilpotentPolynomialSzemeredi}.\hfill$\square$
\end{proof}

\begin{theorem}
\label{ApplicationForSingleTMain}
Let $(X,\mathscr{B},\mu)$ be a probability space and $T,S:X\rightarrow X$ be measure preserving automorphisms for which $S$ is totally ergodic. Let $p_1,\cdots,p_K \in \mathbb{Q}[x]$ be a collection of integer polynomials such that $\{p_1(n+h)-p_1(n),p_2(n+h)-p_1(n),\cdots,p_K(n+h)-p_1(n),p_2(n)-p_1(n),\cdots,p_K(n)-p_1(n)\}$ is independent for all $h \in \mathbb{N}$. For any $f,g_1,\cdots,g_K \in L^\infty(X,\mu)$ for which $\mu_{f,T}\perp\mathcal{L}$, we have 

\begin{equation}
    \lim_{N\rightarrow\infty}\frac{1}{N}\sum_{n = 1}^NT^nf\prod_{i = 1}^KS^{p_i(n)}g_i = \mathbb{E}[f|\mathcal{I}_T]\prod_{i = 1}^K\int_Xg_id\mu,
\end{equation}
with convergence taking place in $L^2(X,\mu)$.
\end{theorem}

\begin{proof}[Proof of Theorem \ref{ApplicationForSingleTMain}]
We see from Lemma \ref{CreatingWeaklyRigidSequences} that $(T^nf)_{n = 1}^\infty$ is spectrally singular, so using standard linearity arguments in conjunction with Lemma \ref{StrongMixingTimesRigidNormAveragesTo0}, it suffices to show that $\left(\prod_{i = 1}^KS^{p_i(n)}g_i\right)_{n = 1}^\infty$ is spectrally Lebesgue if $\int_Xg_id\mu = 0$ for some $1 \le i \le K$. To this end, we first observe that for $h \in \mathbb{N}$ we have

\begin{alignat*}{2}
&\lim_{N\rightarrow\infty}\frac{1}{N}\sum_{n = 1}^N\int_X\left(\prod_{i = 1}^KS^{p_i(n+h)}g_i\right)\left(\prod_{i = 1}^KS^{p_i(n)}\overline{g_i}\right)d\mu\numberthis\\
=&\lim_{N\rightarrow\infty}\frac{1}{N}\sum_{n = 1}^N\int_X\left(\prod_{i = 1}^KS^{p_i(n+h)-p_1(n)}g_i\right)\left(\prod_{i = 2}^KS^{p_i(n)-p_1(n)}\overline{g_i}\right)\overline{g_1}d\mu\\
=&\left(\prod_{i = 1}^K\int_Xg_id\mu\right)\left(\prod_{i = 2}^K\int_X\overline{g_i}d\mu\right)\int_X\overline{g_1}d\mu = 0,
\end{alignat*}
where the penultimate equality follows from Theorem \ref{PolynomialAveragesConvergeForSingleT}, so the desired result now follows from Theorem \ref{StrongMixingClassicalvdC}.
\end{proof}

\begin{remark}
While the conditions imposed on $p_1,\cdots,p_K$ in Theorem \ref{ApplicationForSingleTMain} are more complicated than those imposed by Theorem \ref{ApplicationForSingleT}, they are satisfied by many new collections, such as for $(p_1(x),p_2(x),p_3(x)) \in \left\{\left(x^2,x^3,x^5\right), \left(x^3,x^4,x^5\right),\left(x^6,x^6+x^4,x^6+x^3\right)\right\}$.
\end{remark}

\begin{theorem}\label{ApplicationWithHardyFunctions}
    Let $a_1,\cdots,a_k \in \mathcal{HF}$ have polynomial growth and assume that if $a$ is a nontrivial $\mathbb{R}$-linear combination of $\{a_i(t)\}_{i = 1}^k\cup\{a_i(t+h)\}$ for some $h \in \mathbb{N}$, then

     \begin{equation}
         \lim_{t\rightarrow\infty}\frac{\left|a(t)-\alpha p(t)\right|}{\log(t)} = \infty\text{ for any }p(t) \in \mathbb{Z}[t], \alpha \in \mathbb{R}. 
     \end{equation}
     Let $(X,\mathscr{B},\mu)$ be a probability space and $T,S:X\rightarrow X$ measure preserving automorphisms for which $S$ is ergodic. For any functions $f,g_1,\cdots,g_k \in L^\infty(X,\mu)$ for which $\mu_{f,T}\perp\mathcal{L}$, we have

     \begin{equation}
        \lim_{N\rightarrow\infty}\frac{1}{N}\sum_{n = 1}^NT^nf\prod_{i = 1}^kS^{\lfloor a_i(n)\rfloor}g_i = \mathbb{E}[f|\mathcal{I}_T]\prod_{i = 1}^k\int_Xg_id\mu.
    \end{equation}
\end{theorem}

\begin{proof}
We see from Lemma \ref{CreatingWeaklyRigidSequences} that $(T^nf)_{n = 1}^\infty$ is spectrally singular, so using standard linearity arguments in conjunction with Lemma \ref{StrongMixingTimesRigidNormAveragesTo0}, it suffices to show that $\left(\prod_{i = 1}^KS^{\lfloor a_i(n)\rfloor}g_i\right)_{n = 1}^\infty$ is spectrally Lebesgue if $\int_Xg_id\mu = 0$ for some $1 \le i \le K$. To this end, we first observe that for $h \in \mathbb{N}$ we have

\begin{alignat*}{2}
&\lim_{N\rightarrow\infty}\frac{1}{N}\sum_{n = 1}^N\int_X\left(\prod_{i = 1}^KS^{\lfloor a_i(n+h)\rfloor}g_i\right)\left(\prod_{i = 1}^KS^{\lfloor a_i(n)\rfloor}\overline{g_i}\right)d\mu = 0\numberthis\\
\end{alignat*}
due to Theorem \ref{JointErgodicityForHardy}, so the desired result now follows from Theorem \ref{StrongMixingClassicalvdC}.    
\end{proof}

\begin{remark}
    A simple example of $a_1,\cdots,a_k \in \mathcal{HF}$ satisfying Theorem \ref{ApplicationWithHardyFunctions} is $a_i(t) = t^{\alpha_i}$ where $1 < \alpha_1 < \alpha_2 < \cdots < \alpha_k \in \mathbb{R}\setminus\mathbb{N}$ satisfy $\alpha_{i+1} > \alpha_i+1$ for all $1 \le i < k$. 
\end{remark}

\begin{theorem}
\label{ApplicationForT1T2Main}
Let $(X,\mathscr{B},\mu)$ be a probability space, $L_1 \in \mathbb{N}\cup\{0\}$, $L_2 \in \mathbb{N}$, and $T,R_1,\cdots,R_{L_1},S,W:X\rightarrow X$ be measure preserving transformations. Suppose $\{R_1,\cdots,R_{L_1},S,W\}$ generate a commutative group, and $S$ is weakly mixing. Let $p_1,\cdots,p_{L_2} \in \mathbb{Q}[x]$ be pairwise essentially distinct integer polynomials, each having degree at least 2. For any $f,h_i, g_j \in L^\infty(X,\mu)$ with $1 \le i \le L_1$, $1 \le j \le L_2$, satisfying $\int_Xg_jd\mu = 0$ for some $1 \le j \le L_2$, and for which $\mu_{f,T}\perp\mathcal{L}$, we have

\begin{equation}
    \lim_{N\rightarrow\infty}\frac{1}{N}\sum_{n = 1}^NT^nf\left(\prod_{i = 1}^{L_1}R_i^nh_i\right)\left(\prod_{j = 1}^{L_2}S^{p_j(n)}W^ng_j\right) = 0,
\end{equation}
with convergence taking place in $L^2(X,\mu)$.
\end{theorem}

\begin{proof}
We see from Lemma \ref{CreatingWeaklyRigidSequences} that $(T^nf)_{n = 1}^\infty$ is spectrally singular, so using Lemma \ref{StrongMixingTimesRigidNormAveragesTo0} it suffices to show that

\begin{equation}\label{ProbablyAKMixingSequence}
    \left(\prod_{i = 1}^{L_1}R_i^nh_i\right)\left(\prod_{j = 1}^{L_2}S^{p_j(n)}W^ng_j\right)
\end{equation}
is spectrally Lebesgue. We proceed by induction on $L_1$, beginning with the base case of $L_1 = 0$. We first observe that the set $H$ of $h \in \mathbb{N}$ for which $p_i(n+h)-p_j(n) = (p_i(n+h)-p_1(n))-(p_j(n)-p_1(n))$ is constant, is a finite set. Furthermore, since the $p_i$s are pairwise essentially distinct and of degree at least 2, there are no $i,j,h \in \mathbb{N}$ for which $p_i(n+h)-p_j(n+h) = (p_i(n+h)-p_1(n))-(p_j(n+h)-p_1(n))$ is constant. It follows that for $h \in \mathbb{N}\setminus H$, the collection $\mathcal{C}(h) := \{p_i(n+h)-p_1(n)\}_{i = 1}^{L_2}\cup\{p_i(n)-p_1(n)\}_{i = 2}^{L_2}$ consists of pairwise essentially distinct polynomials, hence

\begin{alignat*}{2}
    &\lim_{N\rightarrow\infty}\frac{1}{N}\sum_{n = 1}^N\left\langle \prod_{j = 1}^{L_2}S^{p_j(n+h)}W^{n+h}g_j, \prod_{j = 1}^{L_2}S^{p_j(n)}W^ng_j\right\rangle\numberthis\\
    = &\lim_{N\rightarrow\infty}\frac{1}{N}\sum_{n = 1}^N\int_X\overline{g_1}S^{p_1(n+h)-p_1(n)}W^hg_1\prod_{j = 2}^{L_2}\left(S^{p_j(n+h)-p_1(n)}W^hg_jS^{p_j(n)-p_1(n)}\overline{g_j}\right)d\mu\\
    =& \prod_{j = 1}^{L_2}\left(\int_XW^hg_jd\mu\int_X\overline{g_j}d\mu\right) = \prod_{j = 1}^{L_2}\left|\int_Xg_jd\mu\right|^2 = 0,
\end{alignat*}
where the second equality follows from the weakly mixing PET, so we may apply Theorem \ref{StrongMixingClassicalvdC} to complete the base case. For the inductive step let us assume that the desired result holds for all $L_1 \le \ell$, and we will show that it also holds for $L_1 = \ell+1$. As before, we see that there is a finite set $H \subseteq \mathbb{N}$, such that $\mathcal{C}(h)$ consists of pairwise essentially distinct polynomials when $h \in \mathbb{N}\setminus H$. Consequently, we see that for such $h$ we have

\begin{alignat*}{2}
    &\lim_{N\rightarrow\infty}\frac{1}{N}\sum_{n = 1}^N\left\langle \left(\prod_{i = 1}^{L_1}R_i^{n+h}h_i\right)\left(\prod_{j = 1}^{L_2}S^{p_j(n+h)}W^{n+h}g_j\right), \left(\prod_{i = 1}^{L_1}R_i^nh_i\right)\left(\prod_{j = 1}^{L_2}S^{p_j(n)}W^ng_j\right)\right\rangle\\
    = &\lim_{N\rightarrow\infty}\frac{1}{N}\sum_{n = 1}^N\int_X\overline{h_1}R_1^hh_1\prod_{i = 2}^{L_1}\left((R_iR_1^{-1})^n\overline{h_i}R_i^hh_i\right)\numberthis\label{IsThisActually0?}\\
    &\textcolor{white}{\lim_{N\rightarrow\infty}\frac{1}{N}\sum_{n = 1}^N\int_X\overline{g_1}R_1^hg_1}\prod_{j = 1}^{L_2}\left(S^{p_j(n+h)}(WR_1^{-1})^nW^hg_jS^{p_j(n)}(WR_1^{-1})^n\overline{g_j}\right)d\mu.
\end{alignat*}
We now observe that the transformations $\{R_2R_1^{-1},\cdots,R_{L_1}R_1^{-1},WR_1^{-1},S\}$ generate a commutative group, so we may apply the inductive hypothesis to conclude that

\begin{equation}
    \left(\prod_{i = 2}^{L_1}\left((R_iR_1^{-1})^n\overline{h_i}R_i^hh_i\right)\prod_{j = 1}^{L_2}\left(S^{p_j(n+h)}(WR_1^{-1})^nW^hg_jS^{p_j(n)}(WR_1^{-1})^n\overline{g_j}\right)\right)_{n = 1}^\infty
\end{equation}
is spectrally Lebesgue, so we may use Corollary \ref{NearlyStronglyMixingAveragesTo0} to show that the quantity in Equation \eqref{IsThisActually0?} is 0, which completes the induction.
\end{proof}

\begin{remark}\label{ZeroEntropyConjectureExplanation}
    We saw in the proof of Theorem \ref{ApplicationForT1T2Main} that the assumption of $S$ being weakly mixing is what enabled us to use the weakly mixing PET, which allowed us to show that most of the multiple correlations of the sequence appearing in Equation \eqref{ProbablyAKMixingSequence} are $0$. The fact that so many of these correlations are $0$ is what led us to Conjecture \ref{ZeroEntropyConjecture}, and of course also leads to conjecture that $T$ in Theorem \ref{ApplicationForT1T2Main} can be assumed to have zero entropy instead of just singular spectrum. We state our final theorem without proof since the proof is similar to that of Theorem \ref{ApplicationForT1T2Main} after replacing the weakly mixing PET with Theorem \ref{WeakMixingHardyPET}. Furthermore, while we state the theorem for $T$ having singular spectrum, we conjecture that it is still true for $T$ having zero entropy.
\end{remark}

\begin{theorem}
\label{ApplicationForT1T2Hardy}
Let $(X,\mathscr{B},\mu)$ be a probability space, $L_1 \in \mathbb{N}\cup\{0\}$, $L_2 \in \mathbb{N}$, and $T,R_1,\cdots,R_{L_1},S,W:X\rightarrow X$ be measure preserving transformations. Suppose $\{R_1,\cdots,R_{L_1},S,W\}$ generate a commutative group, and $S$ is weakly mixing. Let $a_1,\cdots,a_{L_2} \in \mathcal{HF}$ have polynomial growth and satisfy

\begin{alignat}{2}
    &\lim_{t\rightarrow\infty}\frac{|a_i(t)|}{\log(t)} = \infty\text{ for all }1 \le i \le L_2\text{, and}\\
    &\lim_{t\rightarrow\infty}\frac{|a_i(t+s)-a_j(t)|}{\log(t)} = \infty\text{ for all }1 \le i \le j \le L_2, s \in \mathbb{N}, (i,s) \neq (j,0).
\end{alignat}
For any $f,h_i, g_j \in L^\infty(X,\mu)$ with $1 \le i \le L_1$, $1 \le j \le L_2$, satisfying $\int_Xg_jd\mu = 0$ for some $1 \le j \le L_2$, and for which $\mu_{f,T}\perp\mathcal{L}$, we have

\begin{equation}
    \lim_{N\rightarrow\infty}\frac{1}{N}\sum_{n = 1}^NT^nf\left(\prod_{i = 1}^{L_1}R_i^nh_i\right)\left(\prod_{j = 1}^{L_2}S^{\lfloor a_i(n)\rfloor}W^ng_j\right) = 0,
\end{equation}
with convergence taking place in $L^2(X,\mu)$.
\end{theorem}

\begin{remark}
    A simple example of $a_1,\cdots,a_k \in \mathcal{HF}$ satisfying Theorem \ref{ApplicationForT1T2Hardy} is $a_i(t) = t^{\alpha_i}$ where $1 < \alpha_1 < \alpha_2 < \cdots < \alpha_k \in \mathbb{R}$.
\end{remark}

\begin{Backmatter}
\bibliographystyle{abbrv}
\begin{center}
	\bibliography{FOMPSA}
\end{center}
\end{Backmatter}

\end{document}